\documentclass[a4paper,UKenglish,cleveref, autoref, thm-restate]{lipics-v2021}
\title{Positivity of Second-Order Nearly Linearly Recurrent Sequences}
\titlerunning{Positivity of Nearly Linearly Recurrent Sequences}

\author{Mahsa Shirmohammadi}{CNRS, IRIF, Universit\'{e} Paris
  Cit\'{e}, Paris, France}{mahsa@irif.fr}
{https://orcid.org/0000-0002-7779-2339}{}
\author{James Worrell}{Department of Computer Science, University of Oxford, UK}{jbw@cs.ox.ac.uk}
{https://orcid.org/0000-0001-8151-2443}{}

\authorrunning{M. Shirmohammadi, and J. Worrell}

\Copyright{Mahsa Shirmohammadi, and James Worrell}

\ccsdesc[500]{Theory of computation~Quantitative automata}
\ccsdesc[500]{Computing methodologies~Number theory algorithms}

 \keywords{Linear Recurrence, Positivity, Linear Time-Invariant Systems, Subspace Theorem, Transcendence}
 \funding{This work was supported by EPSRC grant BeLinDySys (EP/N008197/1) and
 ANR grant VeSyAM (ANR-22-CE48-0005)}

\hideLIPIcs

\def\N{\mathbb{N}}

\nolinenumbers

\begin{document}

\maketitle

\begin{abstract}
Nearly linear recurrences generalise linear recurrences and can be represented as special cases of both linear time-invariant systems in control theory and linear-constraint loops in program analysis. We formulate the Positivity Problem for such recurrences: given a recurrence and initial values, decide whether every sequence satisfying the recurrence is termwise nonnegative. This problem generalises the Positivity Problem for linear recurrence sequences and is a special case of halfspace non-reachability for linear time-invariant systems. Our main result is a decision procedure for order-2 recurrences. The termination of our procedure relies on a transcendence theorem of independent interest: we prove that certain convergent series obtained by summing the absolute values of terms of algebraic linear recurrence sequences are transcendental.
\end{abstract}

\paragraph*{AI statement.}
The authors used ChatGPT to assist with proofreading, calculations in Example 4, and simplifying the proof of Theorem 13.

\section{Introduction}
\subsection{Main Results}
Let $\boldsymbol u = (u_n)_{n=0}^\infty$ be a sequence of real numbers.
A \emph{nearly linear recurrence} of order $d$ is a system of inequalities 
\begin{gather}
     \varepsilon_0 \leq u_{n+d} - a_{d-1}u_{n+d-1}-\cdots - a_{0}u_n \leq \varepsilon_1 \qquad (n\in \mathbb N)\, ,
\label{eq:nearly}
\end{gather}
where $a_0,\ldots,a_{d-1}$ and $\varepsilon_0 \leq \varepsilon_1$ are real algebraic numbers, with $a_0\neq 0$.
Recall here that an algebraic number is a complex number that is the root of a non-zero polynomial with integer 
coefficients.
Nearly linear recurrences generalise the classical notion of linear recurrences:
if $\varepsilon_0=\varepsilon_1=0$, then~\eqref{eq:nearly} specifies a linear recurrence of order $d$, while
if we merely have $\varepsilon_0=\varepsilon_1$, then~\eqref{eq:nearly} specifies 
an inhomogeneous linear recurrence of order $d$.  
Carrying over standard terminology from linear recurrences,
we define the \emph{characteristic polynomial} of~\eqref{eq:nearly} to be
$f(x):=x^d - a_{d-1}x^{d-1} - \cdots - a_0$ and the \emph{characteristic roots} 
to be the roots of $f(x)$.  A sequence $\boldsymbol u$ satisfying~\eqref{eq:nearly} is called a \emph{nearly linearly recurrent sequence (NLRS)}.  

The notion of NLRS was introduced by Akiyama, Evertse, and
Peth\H{o}~\cite{Akiyama2017}, motivated by the study of shift radix
numeration systems and discretised rotations in the
plane~\cite{Akiyama2013}.  NLRS are considered in~\cite{Akiyama2017}
as LRS subject to bounded disturbances but, as we explain below, 
they are also
special cases of (discrete-time) linear time-invariant (LTI) systems in control theory 
and linear constraint loops in program analysis.


We study the Positivity Problem for NLRS, where the input is a
recurrence of the form~\eqref{eq:nearly} 
and real algebraic initial values $u_0,\ldots,u_{d-1}$, and the task is
to determine whether every sequence $\boldsymbol u$
satisfying~\eqref{eq:nearly} with the specified initial values
satisfies $u_n \geq 0$ for all $n\in\mathbb N$. (Here and throughout, a sequence is called \emph{positive} if it is termwise 
nonnegative.)
Evidently the
Positivity Problem for NLRS generalises the Positivity Problem for
linear recurrence sequences (LRS): \emph{given an LRS $\boldsymbol u$,
  determine whether $u_n\geq 0$ for all $n$}.\footnote{More precisely, the 
  Positivity Problem for linear recurrences of order $d$ reduces to
  the Positivity Problem for NLRS of order $d-1$.  Indeed, a nonzero nonnegative LRS must have 
  a positive real dominant characteristic root $\rho$.  Scaling by $\rho^{-n}$ preserves signs and makes $1$
  a characteristic root. Factoring the resulting characteristic polynomial
  as $(x-1)Q(x)$ yields an inhomogeneous recurrence of order at most $d-1$, which is a special case 
  of a nearly linear recurrence of order $d-1$.}  The decidability of the latter
was posed as an open question some 50 years ago~\cite[Section
II.12]{SalomaaSoittola1978} in the context of weighted automata and
formal power series.  
Decidability is known for LRS of order at most 5, but remains open in
general~\cite{OuaknineW14}.  
As shown in Section~\ref{sec:dec}, the Positivity Problem for NLRS of order $d$ is a special case of the
non-reachability problem for LTI systems in
dimension $d$.

The main result of this paper is a decision procedure for the
Positivity Problem for nearly linear recurrences of order 2. 
The main obstacle is the stable oscillatory case, namely when the characteristic roots of~\eqref{eq:nearly}
are complex and of absolute value strictly less than one.
In this case we reduce positivity to a threshold problem---determine  whether $\sum_{m=0}^\infty |u_m|>\alpha$, where $\alpha$ is algebraic and
 $(u_m)_{m=0}^\infty$ is a real algebraic LRS of order 2 whose characteristic roots are those of the recurrence~\eqref{eq:nearly}.
We show that the left-hand sum is transcendental (and hence not equal to $\alpha$) unless 
the sign of $u_m$ is ultimately periodic.  This allows us to solve the threshold problem by numerical approximation of the left-hand sum.
However, since the transcendence result is qualitative, we do not obtain complexity bounds on our 
decision procedure.

There are several other natural decision problems concerning the positivity of
NLRS that also specialise known problems on LTI
systems and linear constraint loops.  We review these problems in Section~\ref{sec:conclusion}.
They provide further evidence that NLRS offer a natural and abstract setting in which to 
study algorithmic problems in control theory and program analysis.

\subsection{Technical Contributions}
We now explain some of the technical contributions of this paper in more detail.
In Section~\ref{sec:dec} we recast the Positivity Problem for order-2 NLRS as a non-reachability problem for planar LTI systems.
An LTI system in
dimension $d$ is specified by a non-deterministic vector recurrence
$\boldsymbol x_{n+1} \in A\boldsymbol x_n + U$, where $A$ is a
$d\times d$ matrix and $U\subseteq \mathbb R^d$ is a polyhedral set~\cite{FijalkowOPP019}.

\begin{example}
\label{ex:gain}
Consider the planar LTI system given by the recurrence
\[ \boldsymbol x_{n+1}=A\boldsymbol x_n+r_n \boldsymbol b, \qquad r_n\in[-1,1], \qquad \boldsymbol x_0=0\, , \]
where
\[
A:=\frac1{10}
\begin{pmatrix}
3 & -4\\
4 & 3
\end{pmatrix}\quad\text{and}
\quad
\boldsymbol b:=\begin{pmatrix}2\\3\end{pmatrix}.
\]
The control set is the one-dimensional polytope
\(U:=\{r\boldsymbol b:r\in[-1,1]\} \).
The matrix $A$ is a rotation through an
angle $\phi\in(0,\pi/2)$, where
$\cos\phi=3/5$ and $\sin\phi=4/5$, followed by a contraction
by a factor of $1/2$.

Fix an output vector $\boldsymbol e_1:=(1,0)^\top$.  The \emph{induced $\ell_\infty$-gain} (see~\cite[Section 5.2.5]{boyd1991linear}) of this system is the supremum output value 
over all runs, that is, the supremum displacement 
in the direction $\boldsymbol e_1$.  
Unfolding the dynamics gives
\(\boldsymbol x_n=\sum_{m=0}^{n-1}r_{n-1-m}A^m \boldsymbol b\).
Since the controls can be chosen independently, the maximum
horizontal displacement at time $n$ is
\[
x^{(\mathrm{max})}_n := \sum_{m=0}^{n-1}\left|\boldsymbol e_1^\top A^m \boldsymbol b\right|.
\]
Indeed, each summand is maximised by choosing $r_{n-1-m}$ to
have the same sign as $\boldsymbol e_1^\top A^m \boldsymbol b$.

Identify $\mathbb R^2$ with $\mathbb C$, and put
\(
\lambda:=\frac{3+4i}{10}\) and \(a:=\frac{2+3i}{2}\). 
Under this identification, $A$ acts as multiplication by
$\lambda$, and $\boldsymbol b$ corresponds to $2a$. Hence
\(\boldsymbol e_1^\top A^m \boldsymbol b
=a\lambda^m+\overline{a\lambda^m}\). 
It follows that the  $\ell_\infty$-gain is
\[ \gamma:= \lim_{n\to\infty}x^{(\mathrm{max})}_n
  =\sum_{m=0}^{\infty}
    \left|a\lambda^m+\overline{a\lambda^m}\right|  \, .
\]
The number $\gamma$ is transcendental by
Theorem~\ref{thm:main}.  This means that we can decide reachability of any halfspace $\{ \boldsymbol x \in \mathbb R^2: \boldsymbol e_1^\top \boldsymbol x\geq c\}$
for real algebraic $c$ by deciding whether $\gamma>c$.  Transcendence of $\gamma$ precludes $\gamma=c$ and hence the inequality can be determined by computing $\gamma$ to sufficient precision.
\end{example}

Our main mathematical contribution concerns the transcendence of a class of numbers that arise as gains of planar LTI systems.
Let $\mathcal S$ be the class of real numbers of the form
$\sum_{m=0}^\infty|u_m|$, where $(u_m)_{m=0}^\infty$ is a non-zero real algebraic LRS of order~2
that converges to zero and whose sign is not ultimately periodic.
Such a sequence has the form
$u_m=a\lambda^m + \overline{a\lambda^m}$, where
$|\lambda|<1$ and
$\lambda/\overline{\lambda}$ is not a root of unity.
Theorem~\ref{thm:main} shows the transcendence of
every element of $\mathcal S$.  This allows us to effectively determine inequalities 
between elements of $\mathcal S$ and algebraic numbers, which was the key step in Example~\ref{ex:gain}.

For an LRS $(u_m)_{m=0}^\infty$ as in the previous paragraph, transcendence of $\sum_{m=0}^\infty |u_m|$ 
in the special case that $\arg(a)$ is a rational multiple of $\pi$ follows from~\cite[Theorem 1(ii)]{LOW}.
While this special case suffices for the application to positivity of order-2 NLRS, the ability to show transcendence of general elements of $\mathcal S$
is a result of independent interest that, among other things, allows 
determining reachability in more general planar LTI systems; see Example~\ref{ex:gain}.  
We reserve a systematic treatment of reachability for planar LTI systems for later work.

Theorem~\ref{thm:main} is based on the novel insight that, although the sequence of absolute values $(|u_m|)_{m=0}^\infty$
is not an LRS, there is an infinite family of 
linear recurrence relations whose exceptional sets (the sets of indices at which the recurrence fails) have arbitrarily small density (see Example~\ref{ex:approximation}). 
This fact can be used not only to show transcendence of individual elements of $\mathcal S$, but also of linear combinations thereof.
For example, in Theorem~\ref{thm:LI} we show that the difference between two elements of $\mathcal S$ is transcendental, 
subject to multiplicative independence of the characteristic roots associated with the LRS.  
This allows us to determine inequalities between two different elements of $\mathcal S$ under the hypotheses of Theorem~\ref{thm:LI}.

\subsection{Related Work}
The paper~\cite{Akiyama2017} provides the first systematic study
of NLRS.  Among other things, this work characterised the possible
asymptotic behaviours of NLRS and showed that the analogue of the
Skolem-Mahler-Lech Theorem (the set of zeros of an LRS is the union of
a finite set and finitely many arithmetic progressions) fails for
NLRS.

A number of recent works consider decision problems for perturbed
versions of linear dynamical systems.  The paper~\cite{AkshayBGV24}
studies the decidability of robust versions of decision problems on
linear recurrence sequences.  That work considers only perturbations
of the initial values of the recurrence, whereas in the present paper
we consider perturbations of the recurrence at each time step.  The
paper~\cite{rounding} considers linear dynamical systems with rounding
after each transition step, in an attempt to study dynamics in the
context of bounded precision arithmetic.  Since the rounding
considered in~\cite{rounding} is deterministic, the transition
relation remains deterministic, unlike in the present paper.  The main
result of~\cite{rounding} is a decision procedure for a variant of
Kannan and Lipton's orbit problem in the presence of rounding for
hyperbolic linear dynamical systems.

There is a rich body of work on the notion of \emph{chain
  reachability} in dynamical systems~\cite[Chapters 6 and
7]{katok1995modern}.  This is a notion of reachability that is stable
under arbitrarily small perturbations of the standard dynamics.  The
paper~\cite{DCostaKMOSS021} shows the decidability of the natural
analogue of the Skolem Problem for LRS with respect to chain
reachability.  By contrast, the present work deals with LRS under a
fixed perturbation range---the interval $[\varepsilon_0,\varepsilon_1]$ in
Equation~\eqref{eq:nearly}.  The resulting notion of NLRS is a strict
generalisation of LRS, and from an algorithmic viewpoint this makes
the decision problems more challenging than for LRS, whereas
decision problems associated with chain reachability appear to be more
tractable and are not generalisations of the classical Skolem and
Positivity Problems for LRS.

The paper~\cite{BL23} shows 
transcendence of Hecke-Mahler
series $\sum_{n=0}^\infty \lambda^n \lfloor n\theta + \psi \rfloor$ for algebraic $\lambda$, $0<|\lambda|<1$, and $\theta,\psi\in(0,1)$ with $\theta$ irrational.
The numbers in the above class $\mathcal S$ can be described as particular linear combinations of values of Hecke-Mahler series.  Thus the transcendence results in Theorem~\ref{thm:main}
and Theorem~\ref{thm:LI} can be recast as statements of $\overline{\mathbb Q}$-linear independence for certain values of Hecke-Mahler series.

\section{Decidability of Positivity at Order Two} 
\label{sec:dec}
In this section we give a decision procedure for the Positivity
Problem for order-2 NLRS.  The procedure depends on a
transcendence result, which is the subject of Sections~\ref{sec:partI}
and~\ref{sec:main}.
We assume that the constants that define
the nearly linear recurrence and the initial values
of the sequence are real algebraic.
We use standard algorithms for exact computation with real algebraic numbers, such as root isolation, arithmetic, comparison, and sign determination; see, e.g.,~\cite{basu2006algorithms}.
We start by recalling some basic facts about
linear recurrence sequences (see~\cite[Chapter 6, Sections 1 and 2]{berstel2011noncommutative} for details).
\subsection{Linear Recurrence Sequences}
\label{sec:LRS}
In this section we consider LRS of real algebraic numbers.  Recall
that every such LRS $\boldsymbol u$ has a closed-form representation
$u_n = \sum_{i=1}^s P_i(n) \lambda_i^n$, where
$\lambda_1,\ldots,\lambda_s$ are roots of the characteristic
polynomial and $P_1,\ldots,P_s$ are non-zero polynomials with
algebraic coefficients.  Conversely, every sequence $\boldsymbol u$
admitting such an exponential-polynomial representation is an LRS.  A
third equivalent characterisation of LRS is via matrix powers: given a
nonsingular
$d\times d$ matrix $A$ and $d$-dimensional vectors
$\boldsymbol x,\boldsymbol y$, with all entries real algebraic, the
sequence $u_n:= \boldsymbol x^\top A^n \boldsymbol y$ is an LRS whose
characteristic roots are all eigenvalues of $A$.  Moreover, if $A$ is
diagonalisable, then the sequence $\boldsymbol u$ admits an
exponential-sum representation $u_n=\sum_{i=1}^s a_i \lambda_i^n$ for
fixed algebraic numbers $a_1,\ldots,a_s$.  Recall also that for any
sequence $\boldsymbol u$ satisfying a linear recurrence with
characteristic polynomial $Q(x)$, the sequence of partial sums
$\boldsymbol v=(v_n)_{n=0}^\infty$, defined by
$v_n:=\sum_{k=0}^{n-1}u_k$, satisfies a linear recurrence with
characteristic polynomial $Q(x)(x-1)$.

Recall that a \emph{dominant characteristic root} of a linear recurrence is one whose absolute value 
is maximum among the absolute values of characteristic roots.
It is folklore that an LRS has ultimately constant sign if its minimum-order recurrence has a strictly 
dominant positive real characteristic root.  The latter property is moreover effective: for a given sequence
we can compute an index beyond which the sign is constant
(see~\cite[Section 4.1]{OuaknineW14}).  Let $\mathsf{PosPow}$ be the collection
of all LRS $\boldsymbol u$ such that either $\boldsymbol u$ is identically zero or
there exists $L \in \mathbb Z_{>0}$ for which  $\lambda^L$ is positive real for every characteristic
root $\lambda$ of $\boldsymbol u$.  
Such a sequence  $\boldsymbol u$ is the interleaving of  LRS
(namely the sequences
$(u_{Ln+\ell})_{n=0}^\infty$ for $\ell=0,\ldots,L-1$), each of which
is either identically zero or has
a dominant real positive characteristic root. 
It follows that the Positivity Problem is decidable for
LRS in $\mathsf{PosPow}$ and that
if $\boldsymbol u \in \mathsf{PosPow}$ then some 
infinite tail of the sequence $(|u_n|)_{n=0}^\infty$ of absolute 
values lies in $\mathsf{PosPow}$.  Furthermore,
by considering the exponential-polynomial representation of LRS, it is
easy to see that $\mathsf{PosPow}$ is closed under
linear combinations, Hadamard (pointwise) products, and partial sums.

We conclude the section with the following observation.  If
$\alpha,\mu \in \mathbb C$, $|\mu|=1$, and $\mu$ is not a
root of unity, then
\begin{gather}
\liminf_{n\rightarrow \infty} (\alpha\mu^n + \overline{\alpha
\mu^n}) = - 2|\alpha| \, .
  \label{eq:liminf}
\end{gather}
Equation~\eqref{eq:liminf} follows immediately from the fact that $\{
\mu^n : n\in \mathbb N\}$ is dense in the set $\{z \in \mathbb C:|z|=1\}$, which is a consequence
of Kronecker's Theorem on inhomogeneous Diophantine approximation.
\subsection{Matrix-Power Formulation of NLRS}
As a first step, we give a matrix-power formulation of the nearly
linear recurrence~\eqref{eq:nearly}.  This shows that the Positivity
Problem for NLRS is a special case of the non-reachability problem for
LTI systems.  To this end, for the companion matrix
\begin{gather}
A :=
\begin{bmatrix}
a_{d-1} & a_{d-2} & \cdots & a_1 & a_0 \\
1 & 0 & \cdots &0 & 0 \\
0 & 1 &  \cdots & 0& 0 \\
\vdots & \vdots & \ddots & \vdots& \vdots \\
0 &  0 & \cdots & 1& 0
\end{bmatrix}\, ,
\label{eq:MAT}
\end{gather}
we consider the vector recurrence 
\begin{gather}
    \boldsymbol x_{n+1} = A\boldsymbol  x_n+r_n \boldsymbol e_1 \quad (n\in \mathbb N,\, r_n \in [\varepsilon_0,\varepsilon_1]) \, ,
\label{eq:nearly-mat}
\end{gather}  
where $\boldsymbol e_1$ denotes the coordinate vector $(1,0,\ldots,0)^\top$.
Clearly, every sequence $(u_n)_{n=0}^\infty$ satisfying~\eqref{eq:nearly} has the form $u_n=\boldsymbol e_d^\top \boldsymbol x_n$ for 
a solution $(\boldsymbol x_n)_{n=0}^\infty$ of~\eqref{eq:nearly-mat} with initial vector $\boldsymbol x_0 = (u_{d-1},\ldots,u_0)^\top$.
Conversely, for every solution $(\boldsymbol x_n)_{n=0}^\infty$ of~\eqref{eq:nearly-mat}, the sequence 
$(u_n)_{n=0}^\infty$ defined by $u_n :=\boldsymbol e_d^\top \boldsymbol x_n$ satisfies~\eqref{eq:nearly}.
Systems of the form~\eqref{eq:nearly-mat}
are known in control theory as  (discrete-time) \emph{linear time-invariant (LTI)
  systems}.
  Borrowing terminology from this area we will call the
scalars $r_n$ the \emph{controls}.  

We can recast the Positivity Problem for NLRS in terms of vector
recurrences of the form~\eqref{eq:nearly-mat}: \emph{for a given
  instance of~\eqref{eq:nearly-mat} and initial vector
  $\boldsymbol x_0 \in (\overline{\mathbb Q}\cap\mathbb R)^d$, does every choice of 
  $r_0,r_1,r_2,\ldots \in [\varepsilon_0,\varepsilon_1]$ yield a solution
  of~\eqref{eq:nearly-mat} that remains in the halfspace
  $\{ \boldsymbol x \in \mathbb R^d : \boldsymbol e_d^\top \boldsymbol
  x \geq 0 \}$}?  The complement of this problem thus asks
whether there exist $n$ and controls
$r_0,\ldots,r_{n-1} \in [\varepsilon_0,\varepsilon_1]$ such that
$\boldsymbol e_d^\top \boldsymbol x_n < 0$.  This is an instance of
the reachability problem for LTI
systems~\cite{AsarinDG03,FijalkowOPP019}, in which the target is a
halfspace and the control set is a bounded line segment.

By unfolding the recurrence,
the solution of~\eqref{eq:nearly-mat} can be written in the form \[\boldsymbol x_n = A^n \boldsymbol x_0+\sum_{k=0}^{n-1} r_{n-1-k}A^k\boldsymbol e_1 \, .\]
We thus focus on the positivity of the following sequence over all $n\in\mathbb N$ and all
control sequences $r_0,r_1,r_2,\ldots \in [\varepsilon_0,\varepsilon_1]$: 
\begin{gather}
  u_n =  \boldsymbol e_d^\top \boldsymbol x_n = \boldsymbol e_d^\top A^n
   \boldsymbol x_0 + \sum_{k=0}^{n-1} r_{n-1-k}
   \boldsymbol e_d^\top A^k \boldsymbol e_1 \, .
\label{eq:main}
\end{gather} 

To further analyse~\eqref{eq:main}, we introduce two auxiliary linear
recurrence sequences, given by matrix-power expressions.
These are the \emph{zero-control LRS} $\boldsymbol
u^{(z)}=\big(u^{(z)}_n\big)_{n=0}^\infty$ and the 
\emph{control LRS}
$\boldsymbol u^{(c)} = (u^{(c)}_n)_{n=0}^\infty$, which are
respectively defined, for all $n \in \mathbb N$, by
\begin{gather}
u^{(z)}_n:=\boldsymbol e_d^\top A^n \boldsymbol x_0\qquad\text{and}\qquad
  u^{(c)}_n := \boldsymbol e_d^\top A^n \boldsymbol e_1\, .
  \label{eq:zero-control}
  \end{gather}
We can rewrite~\eqref{eq:main} in terms of the zero-control and control LRS as follows:
\begin{gather}
   u_n = \boldsymbol e_d^\top \boldsymbol x_n = u_n^{(z)} + \sum_{k=0}^{n-1}
   r_{n-1-k} u_k^{(c)} \qquad (r_0,r_1,r_2,\ldots \in [\varepsilon_0,\varepsilon_1])\, .
    \label{eq:main2}
\end{gather}

For a given $n\in\mathbb N$,
we can minimise the sum in~\eqref{eq:main2} by minimising each summand individually.  
 The choice of controls $r_0,\ldots,r_{n-1}\in
 [\varepsilon_0,\varepsilon_1]$ that achieves this is
\begin{gather}r_{n-1-k}:= \begin{cases} \varepsilon_1 & \text{if $u_k^{(c)}
      <0$}\\
    \varepsilon_0 &  \text{if $u_k^{(c)}
      \geq 0$} \end{cases} \qquad (k\in\{0,\ldots,n-1\}) \, .
\label{eq:opt}
\end{gather}
(The fact that we always choose an endpoint of the interval 
 $[\varepsilon_0,\varepsilon_1]$ as control in~\eqref{eq:opt} is an instance of the 
 so-called \emph{bang-bang principle} for reachability objectives, in which one always chooses an extreme 
 value of the control set.)
Equation~\eqref{eq:opt} gives \[ r_{n-1-k} u^{(c)}_{k}=\left(\frac{\varepsilon_0+\varepsilon_1}{2} \right)u_k^{(c)} +
\left(\frac{\varepsilon_0-\varepsilon_1}{2}\right) |u_k^{(c)}| \, .\]
 Thus, for $n\in \mathbb N$, the minimum of~\eqref{eq:main2} over all possible choices 
of $r_0,\ldots,r_{n-1}\in [\varepsilon_0,\varepsilon_1]$ is given by the expression
\begin{gather}
u^{(min)}_n:= u_n^{(z)} + 
\sum_{k=0}^{n-1} \left(\left(\frac{\varepsilon_0+\varepsilon_1}{2} \right) u_k^{(c)} + \left(\frac{\varepsilon_0-\varepsilon_1}{2}
\right) |u_k^{(c)}| \right)\qquad (n\in\mathbb N) \, .
    \label{eq:master}
\end{gather}
Determining positivity of~\eqref{eq:nearly} thus reduces to deciding positivity of $\boldsymbol u^{(min)}=(u^{(min)}_n)_{n=0}^\infty$.
Note that $\boldsymbol u^{(min)}$ is the pointwise minimum of all sequences satisfying the recurrence~\eqref{eq:nearly} with the specified initial values, but does not typically itself satisfy~\eqref{eq:nearly}.
\subsection{Decision Procedure}
\label{sec:decide}

\begin{theorem}
    The Positivity Problem for NLRS is decidable for order-2 recurrences.
\end{theorem}
\begin{proof}
  We have observed that if $\varepsilon_0=\varepsilon_1$, then a
  sequence satisfying an order-$2$ nearly linear
  recurrence~\eqref{eq:nearly} is an LRS of order at most $3$.  Since the
  Positivity Problem is decidable for order-3 LRS by~\cite[Theorem
  5]{OuaknineW14}, we may henceforth assume that
  $\varepsilon_0<\varepsilon_1$.

  From the previous section,
  our task is to determine the positivity of the sequence
  $\boldsymbol u^{(min)}$ in~\eqref{eq:master} in the case that the
  matrix $A$ in~\eqref{eq:MAT} is $2 \times 2$.  
We divide the proof into four cases according to the spectrum of $A$.  

\emph{Case (i): Some power of $A$ has positive real spectrum.}
Suppose that $A^M$ has positive real spectrum for some $M\geq 1$.  
Then the control LRS $\boldsymbol u^{(c)}$ and zero-control LRS $\boldsymbol u^{(z)}$ both lie
in the class $\mathsf{PosPow}$, defined in Section~\ref{sec:LRS}.
From the closure properties of $\mathsf{PosPow}$, it follows that we can effectively compute $n_0\in \mathbb N$ such that $(|u_n^{(c)}|)_{n=n_0}^\infty$, and thus also
$(u^{(min)}_n)_{n=n_0}^\infty$, lie in $\mathsf{PosPow}$; hence we can determine whether $\boldsymbol u^{(min)}$ is positive.

Having disposed of Case~(i),
the remaining cases assume that $A$ has two complex eigenvalues $\lambda$ and $\overline{\lambda}$ such that $\lambda^n\not\in\mathbb R$ for all $n>0$.  
This allows us to write $u_n^{(c)}=a\lambda^n+\overline{a\lambda^n}$ and 
$u_n^{(z)}=b\lambda^n+\overline{b\lambda^n}$ for algebraic numbers $a,b$ such that $a$ is non-zero.  In particular, 
simplifying~\eqref{eq:master} by evaluating $\sum_{k=0}^{n-1}u_k^{(c)}$, for all $n\in \mathbb N$ we have
\begin{gather}
u_n^{(min)} = c+d\lambda^n+\overline{d\lambda^n} + \left(\frac{\varepsilon_0-\varepsilon_1}{2}\right) \sum_{k=0}^{n-1}|u^{(c)}_k| \, ,
    \label{eq:master-comp}
\end{gather}
where, writing
$\gamma:=\frac{a(\varepsilon_0+\varepsilon_1)}{2(1-\lambda)}$,  we put $c
:= \gamma+\overline{\gamma}$ and  $d:=b-\gamma$. 

\emph{Case (ii): $A$ has spectral radius $>1$}.  In this case we have $|\lambda|>1$.
If $d=0$, then the right-hand side of~\eqref{eq:master-comp} diverges to minus infinity and hence $\boldsymbol u^{(min)}$ is not positive.
Suppose then that $d\neq 0$.  By~\eqref{eq:master-comp}, for $n \in \mathbb N$ we have 
\begin{eqnarray*}
    u_n^{(min)} & \leq & c+ d\lambda^n + \overline{d\lambda^n}\\
    &=& c + |\lambda|^n \left[d \left(\frac{\lambda}{|\lambda|}\right)^n+ \overline{d} \left(\frac{\overline \lambda}{|\lambda|}\right)^n\right] \, .
\end{eqnarray*}
By~\eqref{eq:liminf} we have $\liminf_{n\rightarrow \infty} \left( d \left(\frac{\lambda}{|\lambda|}\right)^n+ \overline{d} \left(\frac{\overline \lambda}{|\lambda|}\right)^n\right)
=-2|d|<0$.  Since $|\lambda|>1$ it follows that 
$\liminf_{n\rightarrow \infty} u_n^{(min)}=-\infty$ and, in particular,
$\boldsymbol u^{(min)}$ is not positive.  

\emph{Case (iii): $A$ has spectral radius $1$.}  Since
$|\lambda|=1$, applying~\eqref{eq:liminf},  we have $\liminf_{n\rightarrow \infty} u_n^{(c)}<0$.  Hence 
$\sum_{k=0}^{n-1} |u^{(c)}_k|$ diverges to infinity as $n \rightarrow \infty$. But  
$c+d\lambda^n+\overline{d}\overline{\lambda}^n$ is bounded.  Inspecting~\eqref{eq:master-comp}, we see that
$\boldsymbol u^{(min)}$ is not positive.

\emph{Case (iv): $A$ has spectral radius $<1$.}
We claim that the limit $\lim_{n\rightarrow\infty} u_n^{(min)}$ exists and is non-zero.  
On the one hand, since $|\lambda|<1$, we have 
\[\lim_{n\rightarrow\infty} \left(u_n^{(z)}+\frac{\varepsilon_0+\varepsilon_1}{2}\sum_{k=0}^{n-1} u_k^{(c)}\right)= \lim_{n\rightarrow\infty} \left(c+d\lambda^n + \overline{d\lambda^n} \right)= c\, .\]
On the other
hand, since $u_n^{(c)} = a\lambda^n + \overline{a\lambda^n}$ for
some non-zero $a\in \overline{\mathbb Q}$, we have that
$\sum_{n=0}^\infty |u_n^{(c)}|$ is transcendental by
Theorem~\ref{thm:main}, below.
Since $\varepsilon_0<\varepsilon_1$ and $\varepsilon_0,\varepsilon_1,c$ are algebraic,
we conclude from~\eqref{eq:master} that $\lim_{n\rightarrow\infty} u_n^{(min)}$
is non-zero.  This proves the claim.  Now we can approximate  this limit numerically to arbitrary precision.
If the limit is negative, then 
$\boldsymbol{u}^{(min)}$ is not positive.  On the other hand, if the limit is positive, then we can compute $n_0$ such that
$u_n^{(min)}>0$ for all $n\geq n_0$. In this case we can determine positivity for $\boldsymbol u^{(min)}$ by checking that 
$u_0^{(min)},\ldots,u_{n_0-1}^{(min)}$ are all nonnegative.

The above case split is effective.  If $A$ has real spectrum then it falls under Case~(i).  Otherwise,
$A$ has complex conjugate eigenvalues $\lambda$ and $\overline{\lambda}$, and Case~(i) applies 
exactly when $\lambda / \overline{\lambda}$ is a root of unity.  But root-of-unity testing and comparison of $|\lambda|$ with 1 are effective, so we can determine 
which among Cases (i)--(iv) applies.
\end{proof}
\subsection{Overview of Theorem~\ref{thm:main} and its Proof}
The number-theoretic ingredient in the preceding decision procedure is
the following theorem:
\begin{restatable}{theorem}{mainthm}
Let $\lambda,a\in\overline{\mathbb Q}$ be non-zero algebraic numbers
such that $|\lambda|<1$ and
$\lambda/\overline{\lambda}$ is not a root of unity. Then
\(
    \alpha
    :=
    \sum_{m=0}^{\infty}
    \left|a\lambda^m+\overline{a\lambda^m}\right|
\)
is transcendental.
\label{thm:main}
\end{restatable}

The proof of Theorem~\ref{thm:main} proceeds by contradiction. Assuming that $\alpha$ is
algebraic, we use continued fractions to construct algebraic
approximations that are too accurate to be consistent with the
$p$-adic Subspace Theorem. The argument has three main steps.  We give a sketch below, before
giving full details in Sections~\ref{sec:partI} and~\ref{sec:main}.

The starting point is to observe that the term
$u_m:=a\lambda^m + \overline{a\lambda^m} $ satisfies an order-2 linear
recurrence with algebraic coefficients.  While the sequence 
of absolute values
$(|u_m|)_{m=0}^\infty$ is
not itself an LRS, we identify a family of linear
recurrences, of successively higher orders, such that the sequence satisfies each recurrence at
all indices outside an exceptional set of low density.
The construction of the recurrences is based
on Diophantine-approximation properties of the number $\theta \in (0,1)$
such that $\lambda = |\lambda|\exp(2\pi i \theta)$.  This number is irrational  
by the assumption that $\lambda/\overline{\lambda}$ is not a root of unity.

Let $q_n$ be a denominator of a convergent of the continued-fraction
expansion of $\theta$. The sequence $(u_m)_{m=0}^\infty$ satisfies the 
recurrence
\[
    u_{m+2q_n}
    -
    \bigl(\lambda^{q_n}+\overline{\lambda}^{\,q_n}\bigr)
    u_{m+q_n}
    +
    |\lambda|^{2q_n}u_m
    =
    0\, .
\]
Since $q_n\theta$ is very close to an integer, the three terms
$u_m,u_{m+q_n},u_{m+2q_n}$ have the same sign for most values of $m\in\mathbb N$.  Hence, outside
a sparse set $\Delta_n$ of exceptional indices,
the same recurrence is satisfied by the sequence $(|u_m|)_{m=0}^\infty$ of absolute values, that is,
\begin{gather}
|u_{m+2q_n}|
    -
    \bigl(\lambda^{q_n}+\overline{\lambda}^{\,q_n}\bigr)
    |u_{m+q_n}|
    +
    |\lambda|^{2q_n}|u_m|
    =
    0 \qquad (m\not\in\Delta_n)\, .
\label{eq:recur-with-error}
\end{gather}
The larger the choice of $n$, the sparser the exceptional set $\Delta_n$.

Second, we sum the above recurrence over $m$. 
If the exceptional set were empty then the sum would give an exact algebraic
relation involving $\alpha$. As it is, the exceptional indices contribute error
terms, and we obtain an identity of the schematic form
\[
    \bigl(1-A_n+B_n\bigr)\alpha
    =
    \nu_n
    +
    \sum_{m\in\Delta_n}w_{n,m},
\]
where $A_n=\lambda^{q_n}+\overline{\lambda}^{\,q_n}$,
$B_n=|\lambda|^{2q_n}$, and $\nu_n$ is a finite initial correction. Each error term $w_{n,m}$ is a short linear
combination of monomials in $\lambda$ and $\overline{\lambda}$, and
its absolute value decays exponentially with $m$. We retain a fixed
number of the first error terms and discard the rest. The
growing gaps between exceptional indices ensure that the sum of the discarded
error terms is extremely small in magnitude. In this way we construct a fixed linear form
$L$, whose coefficients include $\alpha$, and a sequence of tuples
$\Lambda_n$ of algebraic numbers such that
\(    |L(\Lambda_n)| \)
is exceptionally small relative to the arithmetic complexity (or height)
of $\Lambda_n$.

The final step rules out such small values. 
We choose a number field and a finite set of places $S$ such that the
monomial coordinates of $\Lambda_n$ are $S$-units---the number-field
analogue of rational numbers whose prime factors belong to a fixed
finite set. A consequence of the Subspace Theorem then says, roughly,
that the values $L(\Lambda_n)$ cannot be as small as above for
infinitely many $n$, unless two of the $S$-unit coordinates have a
quotient belonging to a fixed finite set.  In the present case this is ruled
out by the assumptions that $|\lambda|<1$
and $\lambda/\overline{\lambda}$ is not a root of unity.
This contradiction shows that $\alpha$ is transcendental.

\section{Approximation by Linear Recurrences}
\label{sec:partI}
In this section we analyse the numbers
$v_m=|a\lambda^m+\overline{a\lambda^m}|$ that appear as summands in the
series in Theorem~\ref{thm:main}.  
While $(v_m)_{m=0}^\infty$ is not an LRS (by~\cite[Lemma 4.7]{Ward21}),
we show that it satisfies a family of linear recurrences outside sparse exceptional sets.

We start by recalling some standard notation.  Every $r \in \mathbb{R}$ can be
written uniquely in the form $r = \lfloor r \rfloor + \{ r\}$, where
$\lfloor r \rfloor \in \mathbb Z$ is the \emph{integer part} of $r$
and $\{ r\} \in [0,1)$ is the \emph{fractional part} of~$r$.  Write
also $\lVert r\rVert:=\min_{m\in\mathbb Z} |r-m|$ for the distance of
$r$ to the nearest integer.  We use the Vinogradov notation
$f(n)\ll g(n)$ to mean $f(n)=O(g(n))$ for
$f,g \colon \mathbb N \to \mathbb R$.

\begin{example}
  Consider the sequence $u_m:=\lambda^m+\overline{\lambda^m}$, where
  $\lambda:=\frac{1}{3}+\frac{i\sqrt{14}}{6}$.  This satisfies the order-2 recurrence
  $u_{m+2}=\frac{2}{3}u_{m+1}-\frac{1}{2}u_m$ for all $m\in\mathbb N$.  
Writing
  $a:=\lambda^{64}+\overline{\lambda^{64}}$ and
  $b:=(\lambda\overline{\lambda})^{64}$, it is easy to verify by direct calculation that $\boldsymbol u$ also
  satisfies the order-128 recurrence 
  \begin{gather}u_{m+128}-a u_{m+64}+b u_m=0\quad (m\in \mathbb N)\, .
  \label{eq:exrecur1}
  \end{gather}

Now we consider at which indices $m\in\N$
the recurrence~\eqref{eq:exrecur1} is satisfied by the sequence $|u_m|$ of absolute values.  The first eight values of $m\in\mathbb N$ 
  at which the recurrence 
  \begin{gather} |u_{m+128}|-a|u_{m+64}|+b|u_m|=0
\label{eq:exrecur}
\end{gather}
fails are $m=11027,11059,11091,11123,22134,22166,22198,22230$.  
To understand this fact, we observe that~\eqref{eq:exrecur} holds for a given 
$m\in\mathbb N$ if $u_m,u_{m+64},u_{m+128}$ all have the same sign, since in this case~\eqref{eq:exrecur1} and~\eqref{eq:exrecur} are equivalent.
Writing $\lambda=\frac{1}{\sqrt{2}}\exp(2\pi i \theta)$, where $\theta \in (0,1)$, we have 
$u_m = 2(\sqrt{2})^{-m}\cos(2\pi  m\theta)$. 
Since
$\cos(x)$ changes sign at $\frac{\pi}{2}$ and $\frac{3\pi}{2}$, Equation~\eqref{eq:exrecur} fails only for values of $m$ at which 
some but not all of  $\{m\theta\},\{(m+64)\theta\},\{(m+128)\theta\}$ lie in $(1/4,3/4)$.
Since $64\theta=11-\varepsilon$, where $\varepsilon\approx  9\times 10^{-5}$,~\eqref{eq:exrecur} fails only for values of $m$ at which 
the orbit $\{m\theta\}$ lands in
$J:=\left(\frac{1}{4},\frac{1}{4}+2\varepsilon\right) \cup \left(\frac{3}{4},\frac{3}{4}+2\varepsilon\right)$.  We can obtain upper and lower 
bounds on the distance between successive $m\in\mathbb N$ for which~\eqref{eq:exrecur} fails by 
analysing the return times of the orbit to $J$.  In the present example, the first-return-time function on 
each of the two intervals composing $J$, considered separately, takes values  
$64,11043,11107$ (the first and third numbers are two consecutive convergent denominators of  the continued fraction expansion of $\theta$ and the middle value 
is their difference).
\label{ex:approximation}
  \end{example}

\subsection{Continued Fractions}
\label{sec:continued}
For the complex number $\lambda$ in Theorem~\ref{thm:main}, let
$\theta \in (0,1)$ be defined by
$\lambda = |\lambda|\exp(2\pi i \theta)$, i.e., $\theta$ is the
argument of $\lambda$ normalised to lie in $(0,1)$.  The assumption
that $\lambda^n\not\in\mathbb R$ for all $n>0$ entails that $\theta$
is irrational.  
In this subsection we consider 
the orbit of a point under the rotation map $R_\theta:[0,1)\rightarrow [0,1)$,
given by $R_\theta(x):=\{x+\theta\}$.
For an interval $J\subseteq [0,1)$, define the map $r_J : J \rightarrow \N$ by $r_J(x):=\min\{ m \geq 1: (R_\theta)^m(x) \in J\}$, giving the number of steps for $x \in J$ to return to $J$ under the action of $R_\theta$.  

We recall some basic facts about continued fractions.  Full details can be found in~\cite[Section 6.2]{Baker2012}.
The simple continued fraction expansion of $\theta$ is the infinite sequence of integers 
\[[0; a_1, a_2, a_3, \ldots]\]
defined by $a_{n+1}:=\left\lfloor \frac{1}{\theta_n} \right\rfloor$, where the $\theta_n$ are
defined inductively by $\theta_0:=\theta$ and
$\theta_{n+1}:=\left\{ \frac{1}{\theta_n} \right\}$ for all $n\in \mathbb N$.  This
sequence is well-defined and infinite if $\theta$ is irrational.
Given $n\in \mathbb{N}$, we write
\[
\frac{p_n}{q_n} =
[0; a_1, a_2, \ldots, a_n]
=
\cfrac{1}{a_1+\cfrac{1}{a_2+\cfrac{1}{\ddots+\cfrac{1}{a_n}}}} 
\]
for the \emph{$n$-th convergent} of the expansion, where $p_n$ and $q_n$ satisfy 
the recurrences 
\[
p_{n+1}=a_{n+1}p_{n}+p_{n-1} \qquad q_{n+1}=a_{n+1}q_{n}+q_{n-1} 
\]
for all $n\in \mathbb N$, with initial conditions $p_{-1}=1,p_0=0,q_{-1}=0,q_0=1$.
Then the sequence $\frac{p_n}{q_n}$ converges to $\theta$.  More precisely, 
writing $\varepsilon_n:=q_n\theta-p_n$, we have $|\varepsilon_n|=\|q_n\theta\|$ for all $n\geq 1$; the sequence $(\varepsilon_n)_{n=0}^\infty$ converges to zero and has
alternating sign.  We moreover have the following \emph{best-approximation property}: a positive integer $q$ lies in the set
$\{q_0,q_1,q_2,\ldots\}$ if and only if $\|q\theta\|<\|q'\theta\|$ for all $0<q'<q$.

\begin{proposition}
    Let $n \in \mathbb N$ and let $J=(\alpha,\beta)\subseteq (0,1)$ be an interval of length $\beta-\alpha > |\varepsilon_n| + |\varepsilon_{n+1}|$.  
    Then for all $x\in J$ we have $r_J(x) \leq q_{n+1}$.
    \label{prop:return}
\end{proposition}
\begin{proof}
For concreteness we give the argument in the case that $\varepsilon_n>0$ and $\varepsilon_{n+1}<0$.  The case that $\varepsilon_n<0$  and $\varepsilon_{n+1}>0$ is symmetric.

Let $x \in J$. 
If $x \in (\alpha,\beta-\varepsilon_n)$,
then $(R_\theta)^{q_n}(x)=\{x+\varepsilon_n\} \in J$.  
If $x \in (\alpha-\varepsilon_{n+1},\beta)$, then $(R_\theta)^{q_{n+1}}(x)=\{x+\varepsilon_{n+1}\} \in J$.
Since $\beta-\varepsilon_{n}>\alpha-\varepsilon_{n+1}$, these two cases are exhaustive.
 \end{proof}

\subsection{A Family of Recurrences}
\label{sec:recurrences}
Let $\lambda$ and $a$ be as in the statement of Theorem~\ref{thm:main}.
Recall that we defined $\theta \in (0,1)$ by putting $\lambda = |\lambda|\exp(2\pi i \theta)$.
Let us likewise define $\psi \in [0,1)$ by putting $a = |a|\exp(2\pi i \psi)$.
We consider the second-order LRS $\boldsymbol u = (u_m)_{m=0}^\infty$, defined by 
\begin{equation}
\begin{aligned}
    u_m &\, := \, a\lambda^m+\overline{a\lambda^m} \\ & \, = \, |2a\lambda^m| \cos(2\pi(m\theta+\psi))\,  .
\end{aligned}
  \label{eq:def-u}
\end{equation}
Since $\theta$ is irrational,
we may assume that $u_m\neq 0$ for all $m\in \mathbb N$ by passing to a tail of the sequence
(which does not affect the transcendence of the sum in Theorem~\ref{thm:main}).

Let $(q_n)_{n=0}^\infty$ be the sequence of denominators of the convergents of 
the continued-fraction expansion of $\theta$ as described in Section~\ref{sec:continued}.
For all $n \in \mathbb N$ write
\begin{equation}
    \begin{aligned}
    A_n &\, :=\, \lambda^{q_n}+\overline{\lambda^{q_n}} \\
    B_n &\, :=\, \lambda^{q_n}\overline{\lambda^{q_n}} \, .
\end{aligned}
\label{eq:anbn}
\end{equation}
An easy verification using the formula $u_m=a\lambda^m +\overline{a\lambda^m}$
shows that for fixed but arbitrary $n\in \mathbb N$, the sequence $\boldsymbol u=(u_m)_{m=0}^\infty$ satisfies the linear recurrence relation
\begin{gather}
u_{m+2q_n}-A_nu_{m+q_n}+B_n u_m=0 \qquad (m\in \mathbb N) \, .
\label{eq:LRS1}
\end{gather}

For each fixed $n\in \mathbb N$, we analyse the set of indices $m\in \mathbb N$ such that the recurrence~\eqref{eq:LRS1} holds for the sequence $|u_m|$ of absolute values:
\[
|u_{m+2q_n}|-A_n|u_{m+q_n}|+B_n |u_m|=0\, .
\]
To this end we define
\begin{gather}
w_{n,m}:=
|u_{m+2q_n}|-A_n|u_{m+q_n}|+B_n|u_m| \qquad (m,n\in\mathbb N)
\label{eq:def-w}
\end{gather}
and let $\Delta_n:=\{m\in\mathbb N: w_{n,m}\neq 0 \}$.  Our objective is to study the structure of $\Delta_n$.

Let $N_0\in\mathbb N$ be such that $|\varepsilon_n|<\frac{1}{8}$ for all $n\geq N_0$.  For $n\geq N_0$, define
\[ J_n:= \begin{cases}
   \bigcup_{\varphi\in \{\frac{1}{4},\frac{3}{4}\}} \left(\varphi-2\varepsilon_n,\varphi\right)& \text{if $\varepsilon_n>0$}\\
    \bigcup_{\varphi\in \{\frac{1}{4},\frac{3}{4}\}}   \left(\varphi,\varphi-2\varepsilon_n\right)  & \text{if $\varepsilon_n<0$}
\end{cases}\]
The following proposition characterises $\Delta_n$ as the set of visit times $m\in \mathbb N$ of the orbit 
$\{m\theta+\psi\}$ to $J_n$.

\begin{proposition}
The following are equivalent for all $m$ and all $n\geq N_0$.
    \begin{enumerate}
        \item $m \in \Delta_n$;
        \item $\{m\theta+\psi\} \in J_n$;
        \item $w_{n,m} \in \{ \pm 2 u_{m+2q_n}, \pm 2B_n u_m \}$.
            \end{enumerate}
    \label{prop:Delta}
\end{proposition}
\begin{proof}
We first show that {\bf 1} implies {\bf 2}.
To this end, notice that if $u_m,u_{m+q_n},u_{m+2q_n}$ all have the same sign, then by~\eqref{eq:LRS1} we have
\[ w_{n,m} = \pm(u_{m+2q_n}-A_nu_{m+q_n}+B_n u_m) = 0 \, .
\]
Suppose that $m\in \Delta_n$.  Then 
$u_{m},u_{m+q_n},u_{m+2q_n}$ do not all have the same sign.
Since $u_m = |2a\lambda^m|\cos(2\pi(m\theta+\psi))$, we have $u_m<0$ 
if and only if $\{ m\theta+\psi \} \in (1/4,3/4)$.  It follows that either one or two of 
$\{m\theta+\psi\}$, $\{m\theta+\psi+\varepsilon_n\}$, $\{m\theta+\psi+2\varepsilon_n\}$ lie in 
the interval $(1/4,3/4)$.
Since $|\varepsilon_n|<1/4$, by the assumption $n\geq N_0$, we conclude 
that $\{m\theta+\psi\} \in J_n$.  

We next show that {\bf 2} implies {\bf 3}.
Suppose that $\{m\theta+\psi\} \in J_n$.  
Since $|\varepsilon_n| < 1/4$, the tuple $(u_m,u_{m+q_n},u_{m+2q_n})$ contains exactly one sign change and hence  its sign vector lies in the set 
 \[ \{ (+,+,-),\;(-,-,+),(+,-,-),(-,+,+) \} \, .\]
 Combined with Equation~\eqref{eq:LRS1}, this entails that $w_{n,m}=\pm 2u_{m+2q_n}$ or $w_{n,m}=\pm 2B_nu_{m}$.

 Finally,  we show that {\bf 3} implies {\bf 1}.
 Suppose that $w_{n,m}=\pm 2u_{m+2q_n}$ or $w_{n,m}=\pm 2B_nu_{m}$.  Then $m\in \Delta_n$ since 
 $B_n\neq 0$ for all $n \in \mathbb N$ and 
 $u_m\neq 0$ for all $m\in \mathbb N$ (by the standing assumption above).
\end{proof}

\begin{proposition}
For all $n\geq N_0$, let
$m_{n,1} < m_{n,2} < \cdots$ be an increasing enumeration of 
$\Delta_n$.  Then:
\begin{enumerate}
    \item  for all $j \geq 0$ we have $\lim_{n\rightarrow \infty}(m_{n,j+1}-m_{n,j})=\infty$, where we define  $m_{n,0}:=0$;
     \item for all $j\geq 1$  
    it holds that  $m_{n,j+4}-m_{n,j}  \geq {q_{n+1}}$; 
    \item for all $j\geq 1$  
    we have $\left\lfloor\frac{j-1}{4} \right\rfloor q_{n+1} \leq m_{n,j} \leq j \, q_{n+1}$.
\end{enumerate}
\label{prop:gap}
\end{proposition}
\begin{proof}
Throughout the proof we use the fact, established in
Proposition~\ref{prop:Delta}, that $m\in \Delta_n$ if and only if $\{m\theta+\psi\} \in J_n$.  
Hence for $m\in\Delta_n$ we have 
$\left\| m\theta+ \psi - \frac{1}{4}\right\| \leq 2|\varepsilon_n|$ or
$\left\| m \theta +\psi  - \frac{3}{4}\right\| \leq 2|\varepsilon_n|$.
It follows that for all $j\geq 1$, 
\begin{gather}
    \| 4(m_{n,j+1}-m_{n,j})\theta\| \leq 16|\varepsilon_n|\, .
\label{eq:stretch}
\end{gather}

We first show Item {\bf 1}.
Since \(u_m\neq 0\), no orbit point \(\{m\theta+\psi\}\) equals \(1/4\) or \(3/4\). 
As the components of \(J_n\) shrink to these two points, every finite initial segment of the orbit 
eventually avoids \(J_n\). Hence \(m_{n,1}\to\infty\). 
For \(j\geq1\), the assertion $\lim_{n\rightarrow \infty}(m_{n,j+1}-m_{n,j})=\infty$
follows from~\eqref{eq:stretch} and the irrationality of \(\theta\).

Next we show Item {\bf 2}.  By the pigeonhole principle, every
five consecutive visits of the orbit $\{m\theta+\psi\}$ to $J_n$ contain two points at distance strictly less than $|\varepsilon_n|$.
By the law of best approximation the number of applications of $R_\theta$ between these two points is at least $q_{n+1}$.

Finally we turn to Item {\bf 3}.  Applying Proposition~\ref{prop:return} to each component of $J_n$, we have  $r_{J_n}(x)\leq q_{n+1}$.  
This shows that $m_{n,j+1}-m_{n,j} \leq q_{n+1}$ for all $j \in \mathbb N$.  It also implies that $m_{n,1}\leq q_{n+1}$ since, if we 
run the orbit backward from its starting point $\psi$ we eventually reach $J_n$,
so 0 lies on an orbit from $J_n$ to $J_n$.
We conclude that $m_{n,j} \leq jq_{n+1}$. 
The lower bound $m_{n,j} \geq \left\lfloor\frac{j-1}{4} \right\rfloor q_{n+1}$ follows from Item {\bf 2}.
\end{proof}

\section{Transcendence Result}
\label{sec:main}
\subsection{Sums of $S$-Units}
\label{sec:sums}
In this subsection we recall some facts about absolute values on
number fields and we state a number-theoretic lemma that will be used to obtain
the transcendence result.  We refer
to~\cite{Bilu08} for a more detailed exposition of the background material.

Let $\mathbb K$ be a finite-dimensional field extension of $\mathbb Q$ and denote by
$\mathcal O_{\mathbb K}$ the subring of $\mathbb K$ consisting of algebraic integers, 
that is, elements in $\mathbb K$ that are roots of
monic polynomials with integer coefficients.  
For non-zero $a \in \mathcal O_{\mathbb K}$ and $\mathfrak p$ a non-zero prime ideal of $\mathcal O_{\mathbb K}$,
we define $\mathrm{ord}_{\mathfrak p}(a):=\max\{ k \in \mathbb N : a \in \mathfrak p^k\}$ to be the order to which $\mathfrak p$ divides $a$.  More generally,
for non-zero $a \in \mathbb K$, where $a = b/c$ for some $b,c\in \mathcal O_{\mathbb K}$, we define 
$\mathrm{ord}_{\mathfrak p}(a):=\mathrm{ord}_{\mathfrak p}(b)-\mathrm{ord}_{\mathfrak p}(c)$.

An absolute value on $\mathbb K$ is a function $|\cdot|:\mathbb K \rightarrow \mathbb R_{\geq 0}$
such that $|a|=0$ if and only if $a=0$, $|ab|=|a||b|$, and $|a+b|\leq |a|+|b|$ for all $a,b\in \mathbb K$.
We say that two absolute values $|\cdot|_1$ and $|\cdot|_2$ on
$\mathbb K$ are \emph{equivalent} if there exists a strictly positive real number $c$
such that $|\cdot|_1=|\cdot|_2^c$. A \emph{place} of $\mathbb K$
is an equivalence class of non-trivial absolute values.  Denote by $M(\mathbb K)$
the set of places of $\mathbb K$.  A classical theorem of Ostrowski shows that each place of $M(\mathbb K)$
is determined either by an embedding of
$\mathbb K$ in $\mathbb C$ or by a
prime ideal of $\mathcal O_{\mathbb K}$.   Specifically, for each place $v\in M(\mathbb K)$ we may
pick a representative absolute value $|\cdot|_v$ as follows, where $|\cdot|$ denotes the standard  absolute value on $\mathbb C$.\footnote{The system below is obtained from
the usual normalised system by raising every absolute value to the
common power \([\mathbb K:\mathbb Q]/2\).  Thus the associated height is
a fixed positive power of the usual projective Weil height.}
\begin{enumerate}[(i)]
    \item if $v$ corresponds to an embedding
$\sigma:\mathbb K\rightarrow \mathbb R$, then define
$|\cdot|_v := |\sigma(\cdot)|^{1/2}$; 
\item if $v$ corresponds to a
complex-conjugate pair of embeddings
$\sigma,\overline{\sigma}:\mathbb K \rightarrow \mathbb C$, then
define
$|\cdot|_v := |\sigma(\cdot)|=|\overline{\sigma}(\cdot)|$;
\item if $v$ corresponds to a prime ideal $\mathfrak p$ of
$\mathcal O_{\mathbb K}$, then write
$|a|_v := N(\mathfrak p)^{-\mathrm{ord}_{\mathfrak p}(a)/2}$ for non-zero
$a \in \mathbb K$, where $N(\mathfrak p)$ is the order of the residue field $\mathcal O_{\mathbb K}/\mathfrak p$.
\end{enumerate}
The above choice of absolute values is such that the \emph{product formula} holds:
for every non-zero
$a\in \mathbb K$ it holds that $\prod_{v \in M(\mathbb K)} |a|_v=1$.

We use the notion of absolute values to define the height of a tuple of elements of $\mathbb K$,
which can be seen as a measure of the arithmetic complexity of such a point.
For $d\geq 2$, the \emph{projective height} function $H:\mathbb K^d \setminus\{0\} \rightarrow\mathbb  R$
is defined by
\[  H(a_1,\ldots,a_d) := \prod_{v \in M(\mathbb K)}
\max(|a_1|_v,\ldots,|a_d|_v) \, .\]

Let $S\subseteq M(\mathbb K)$ be a finite set of places that includes all
Archimedean places. 
The ring $\mathcal O_S$ of \emph{$S$-integers}
is defined by
\[ \mathcal O_S :=\{a\in\mathbb K : |a|_v\leq 1 \text{ for all $v\in
M(\mathbb K)\setminus S$} \}  \, .\]
The units of $\mathcal O_S$ are the elements $a \in \mathbb K$ such that $|a|_v=1$ for all $v \in M(\mathbb K)\setminus S$.
We call such elements \emph{$S$-units}.

We can now state the main result of this subsection, whose 
proof, an application of the Subspace Theorem, is given in Appendix~\ref{sec:proofs}.  Intuitively the result gives  a lower bound on 
the approximation error of 
$S$-integers by linear combinations of $S$-units.

\begin{lemma}
Let $S\subseteq M(\mathbb K)$ be a finite set of places that includes all
Archimedean places.
Suppose that $v_0 \in S$, $d\geq 2$,
  $\alpha_0,\ldots,\alpha_d \in \mathbb K \setminus\{0\}$, and $\varepsilon>0$.
  Then there is a finite set $U\subseteq \mathbb K$ such that for every solution of the inequalities
\begin{eqnarray*}| \alpha_0 x_0 + \cdots + \alpha_dx_d|_{v_0} &<& H(x_1,\ldots,x_d)^{-\varepsilon} \, |x_d|_{v_0}  \\
|x_d|_{v_0} \,  \prod_{v \in S \setminus \{v_0\}} |x_0|_v&<& H(x_0,\ldots,x_{d-1})^{-\varepsilon} 
\end{eqnarray*}
in $S$-units $x_1,\ldots,x_d$ and $x_0 \in \mathcal O_S \setminus \{0\}$, 
we have $x_i/x_j \in U$ for some $i\neq j$ in $\{1,\ldots,d\}$.
\label{lem:s-unit-plus}
\end{lemma}

\subsection{Proof of Theorem~\ref{thm:main}}
\label{sec:partII}
In this section we restate and prove Theorem~\ref{thm:main}.  Throughout $|\cdot|$ denotes the standard absolute value on $\mathbb C$.
\mainthm*
\begin{proof}
We start by recalling relevant notation from Section~\ref{sec:partI}.  For the
algebraic number $\lambda=|\lambda|\exp(2\pi i \theta)$ we have the sequence $(q_n)_{n=0}^\infty$ of
denominators of convergents of the continued-fraction expansion of
$\theta$.  We have also the
LRS $u_m := a\lambda^m+\overline{a\lambda^m}$, the sequences
$A_n:=\lambda^{q_n}+\overline{\lambda^{q_n}}$ and
$B_n:=|\lambda|^{2q_n}$, and the
terms $w_{n,m}:=|u_{m+2q_n}|-A_n|u_{m+q_n}|+B_n|u_m|$ (defined
in~\eqref{eq:def-w}).  For all $n \in \mathbb N$, let
$0 \leq m_{n,1}< m_{n,2} < \cdots$ be an increasing enumeration of the
set $\{m:w_{n,m}\neq 0\}$.  By Proposition~\ref{prop:Delta} we have $w_{n,m_{n,j}} = 2a\xi_{n,j}+\overline{2a\xi_{n,j}}$ for all sufficiently large $n\in \mathbb N$ and all $j\in \{1,2,\ldots\}$,
where
\begin{gather}
\xi_{n,j} \in \left\{  \pm \lambda^{2q_n+m_{n,j}} , \pm \overline{\lambda^{q_n}}\lambda^{q_n+m_{n,j}} \right\} \, .
    \label{eq:def-xi}
\end{gather}

Our task is to show that
$\alpha := \sum_{m=0}^\infty |u_m|$ is transcendental.  
We proceed by assuming that $\alpha$ is algebraic and aim for a contradiction using Lemma~\ref{lem:s-unit-plus}. 
To this end,
let $\mathbb K$ be the subfield of $\mathbb C$ generated over $\mathbb Q$ by the elements $\{\alpha,\lambda, |\lambda|, a, \overline{a}\}$,  
and let $S \subseteq M(\mathbb K)$  be a finite set of places of $\mathbb K$ that contains all Archimedean places 
such that $a,\overline{a}$ are $S$-integers and  $\lambda$ and $\overline{\lambda}$ are $S$-units of $\mathbb K$.
Let $v_0 \in M(\mathbb K)$ denote the place corresponding 
to the standard absolute value $|\cdot|$ on $\mathbb K$ (we regard $\mathbb K$ as a subfield of $\mathbb C$).

For all $n\in \mathbb N$ we have the following equation, where $\nu_n$ denotes the underbraced finite sum:
\begin{eqnarray*}
    \alpha - A_n \alpha + B_n \alpha &=& 
    \underbrace{\sum_{m=0}^{2q_n-1} |u_m| -\sum_{m=0}^{q_n-1} A_n|u_m|}_{=:\nu_n}+
    \sum_{m=0}^\infty \left(|u_{m+2q_n}|-A_n|u_{m+q_n}|+B_n|u_m| \right)\\
    &=& \nu_n +\sum_{m=0}^\infty w_{n,m}\\
    &=& \nu_n + \sum_{j=1}^\infty w_{n,m_{n,j}} \, .
\end{eqnarray*}
We truncate the above sum after the first $t$ terms, where $t$ will be specified later.  For all $j\geq t+1$,
since $m_{n,j} \geq m_{n,t+1}+j-t-1$, we have
\[ |w_{n,m_{n,j}}| \leq |4a\lambda^{2q_n+m_{n,j}}| \leq |4a\lambda^{2q_n+m_{n,t+1}}| |\lambda|^{j-t-1}\, .\] Since also
$\lim_{n\rightarrow \infty} (m_{n,t+1}-m_{n,t})=\infty$, for sufficiently large $n\in \mathbb N$ we obtain
\begin{equation}
      \left |\nu_n-\alpha + A_n\alpha - B_n\alpha + \sum_{j=1}^t w_{n,m_{n,j}}\right| \,= \, \left| \sum_{j=t+1}^\infty w_{n,m_{n,j}}\right|
\, < \, |\lambda|^{m_{n,t}+2q_n} \, .
\label{eq:small}
\end{equation}
Now consider the linear form 
\[L(x_0,\ldots,x_{4+2t}):= x_0 - \alpha x_1 +\alpha x_2+\alpha x_3-\alpha x_4+  \sum_{i=3}^{2+t} (2ax_{2i-1}+\overline{2a}x_{2i})  \]
and sequence of tuples $\Lambda_n=(\Lambda_{n,0},\ldots,\Lambda_{n,4+2t})$, $n \in \mathbb N$, where
\[ \Lambda_n := \left(\nu_n,1,\lambda^{q_n},\overline{\lambda^{q_n}},\lambda^{q_n}\overline{\lambda^{q_n}},\xi_{n,1},\overline{\xi_{n,1}},\ldots,\xi_{n,t},\overline{\xi_{n,t}} 
 \, \right) \in (\mathcal O_S)^{5+2t}\, . \]
 From the assumption that $\alpha$ is algebraic, all coefficients of $L$ are algebraic.  Furthermore, from~\eqref{eq:small} we have 
 \begin{gather}
|L(\Lambda_n)| < |\lambda|^{m_{n,t}+2q_n} \, .
     \label{eq:Lsmall}
 \end{gather}
 
 Note that $\nu_n \in \mathcal O_S$.  Also, from~\eqref{eq:small}, since $A_n$, $B_n$, $w_{n,m_{n,1}},\ldots,w_{n,m_{n,t}}$
 all converge to $0$ as $n\rightarrow \infty$, we have
 $\lim_{n\rightarrow\infty} \nu_n=\alpha$.  Hence $\nu_n\neq 0$ for 
 all sufficiently large $n$. Furthermore,
elementary calculations using the triangle inequality show that there is a constant $c_1>1$ such that for all $n\in\mathbb N$,
\begin{gather}
\prod_{v\in S \setminus \{v_0\} } |\nu_n|_v \leq c_1^{q_n} \,  .
\label{eq:BOUND2}
\end{gather}
Likewise, there is a constant $c_2 >1$ such that for all $n\in\mathbb N$,
\begin{gather}
H(\Lambda_n) \leq c_2^{m_{n,t}+2q_n} \, . 
\label{eq:HEIGHT}
\end{gather}

Let \(B:=\frac{\log c_1}{\log|\lambda|}<0\).
We define $t:=s+4$, where $s \in \mathbb N$ is sufficiently large such that 
\begin{gather}
{\left\lfloor\frac{s-1}{4}\right\rfloor+B}
\ge
{\frac{s+6}{8}} = {\frac{t+2}{8}} \,  .
\label{eq:choice}
\end{gather}
Our aim is to apply Lemma~\ref{lem:s-unit-plus} to the linear form $L$, with distinguished $S$-unit 
coordinate $\Lambda_{n,3+2s}=\xi_{n,s}$, meaning that $|\Lambda_{n,3+2s}|=|\lambda^{m_{n,s}+2q_n}|$.  The following two claims verify the hypotheses 
of the lemma.
\begin{claim}
\label{claim:one}
Put
\(
\varepsilon_1:=
-\frac{1}{8}\frac{\log|\lambda|}{\log c_2}>0
\).
Then
\(
|\Lambda_{n,3+2s}| \cdot 
\prod_{v\in S\setminus\{v_0\}}|\nu_n|_v
\le
H(\Lambda_n)^{-\varepsilon_1}
\) for all sufficiently large $n$.
\end{claim}
\begin{proof}
We have 
\(
|\Lambda_{n,3+2s}|=|\lambda|^{m_{n,s}+2q_n}.
\)
Thus, by~\eqref{eq:BOUND2} and~\eqref{eq:HEIGHT}, to prove the claim it suffices to show that
\[
|\lambda|^{m_{n,s}+2q_n}c_1^{q_n}
\le
c_2^{-\varepsilon_1(m_{n,t}+2q_n)}\, .
\]
Taking logarithms and using the definition of \(\varepsilon_1\) and $B$, this is equivalent to
\begin{gather}
m_{n,s}+2q_n+Bq_n
\ge
\frac{1}{8}(m_{n,t}+2q_n) \, .
\label{eq:desired}
\end{gather}
Recall that \(q_n\le q_{n+1}\) and \(B<0\).  Since, by Proposition~\ref{prop:gap}(3), we also have
$m_{n,s}\ge \left\lfloor\frac{s-1}{4}\right\rfloor q_{n+1}$, it follows that 
\[
m_{n,s}+2q_n+Bq_n
\ge
\left(\left\lfloor\frac{s-1}{4}\right\rfloor+B\right)q_{n+1}.
\]
Likewise, since $m_{n,t}\leq t q_{n+1}$ by Proposition~\ref{prop:gap}(3), we also have
\[
m_{n,t}+2q_n\le (t+2)q_{n+1} \, .
\]
The desired inequality~\eqref{eq:desired} now follows from~\eqref{eq:choice}.
\end{proof}

\begin{claim}
Put $\varepsilon_2:=-\frac{1}{t+2}\frac{\log|\lambda|}{\log c_2}$.  
Then $|L(\Lambda_n)| < H(\Lambda_n)^{-\varepsilon_2} |\Lambda_{n,3+2s}|$ for $n \in \mathbb N$ sufficiently large.
    \label{eq:claim2}
\end{claim}
\begin{proof}
We have $|\Lambda_{n,3+2s}|=|\lambda|^{m_{n,s}+2q_n}$ and,
by~\eqref{eq:Lsmall}, $|L(\Lambda_n)| < |\lambda|^{m_{n,t}+2q_n}$ for sufficiently large $n$.
Hence, in view of~\eqref{eq:HEIGHT}, it suffices to show that 
\[|\lambda|^{m_{n,t}+2q_n} \leq c_2^{-\varepsilon_2(m_{n,t}+2q_n)} |\lambda|^{m_{n,s}+2q_n}\, .\]
Taking logarithms, this is equivalent to 
\[
(m_{n,t}-m_{n,s}) \log |\lambda| \leq - \varepsilon_2 (m_{n,t}+2q_n) \log c_2 \, .
\]
Since $m_{n,t}\leq t q_{n+1}$ (by Proposition~\ref{prop:gap}(3)) and $q_n \leq q_{n+1}$, it suffices to show that 
\[ (m_{n,t}-m_{n,s}) \log |\lambda| \leq -\varepsilon_2 (t+2) q_{n+1} \log c_2 \, .
\]
From the definition of $\varepsilon_2$, it suffices to show that $m_{n,t}-m_{n,s} \geq q_{n+1}$.  But 
this holds for all sufficiently large $n \in \mathbb N$ by Proposition~\ref{prop:gap}(2), since $t=s+4$.
\end{proof}

Claims~\ref{claim:one} and~\ref{eq:claim2} show that, setting $\varepsilon:=\min(\varepsilon_1,\varepsilon_2)/2$,  
the tuple $\Lambda_n$ and linear form $L$ satisfy the hypotheses of Lemma~\ref{lem:s-unit-plus} for infinitely many $n$.
Here we reorder the \(S\)-unit coordinates of \(\Lambda_n\) such that  \(\Lambda_{n,3+2s}\) plays the role of \(x_d\) in Lemma~\ref{lem:s-unit-plus}.  Hence there exist $i,j\geq 1$ with 
$i\neq j$ and a finite set $U$ such that $\Lambda_{n,i}/\Lambda_{n,j} \in U$ for infinitely many $n$. 
Recalling that $\lim_{n\rightarrow\infty} q_n = \infty$ and $\lim_{n\rightarrow\infty} m_{n,j+1}-m_{n,j}=\infty$ for all $j\geq 0$, 
there exists a sequence $(r_n)_{n=0}^\infty$ of integers such that $\lim_{n\rightarrow\infty}|r_n|=\infty$
and for infinitely many $n$ we have
either $\Lambda_{n,i}/\Lambda_{n,j} = \pm (\lambda/\overline{\lambda})^{r_n}$ 
or  $|\Lambda_{n,i}/\Lambda_{n,j}|=|\lambda|^{r_n}$.
Since $|\lambda|<1$ and $\lambda/\overline{\lambda}$ is not a root of unity,
$\Lambda_{n,i}/\Lambda_{n,j}$ can only lie in $U$ for finitely many $n$.  This is a contradiction.
\end{proof}

\section{Conclusion}
\label{sec:conclusion}
Nearly linear recurrences are a generalisation of linear recurrences and 
special cases of linear constraint loops and LTI systems.   
We have shown decidability of the Positivity Problem for NLRS
of order 2. This gives a decision procedure for halfspace
non-reachability for a particular class of planar LTI systems
with line-segment controls. The transcendence theorem also
applies beyond this class, as illustrated by
Example~\ref{ex:gain}.  Indeed,
we believe that our approach can be generalised to halfspace non-reachability in the plane for LTI systems 
with arbitrary control polyhedra.  However, we expect that it will be difficult to handle more general targets (e.g., point targets)
because, in the general case, a separating hyperplane between the reachable set and the target may not have a rational or algebraic description.  
Deciding positivity of order-3 NLRS remains open; specifically, the case with 3 characteristic roots of the same modulus, greater than one, seems 
difficult to handle.  

Positivity problems for NLRS come in several variants, according to whether initial values are specified and whether the objective is to determine  
the positivity of \emph{some} or \emph{all} sequences satisfying a given recurrence.  
While superficially similar,
these variants seem to be fundamentally different from an algorithmic point of view.
In this paper we studied the \emph{initialised, universal} variant: 
determine the positivity of every sequence satisfying a given 
recurrence~\eqref{eq:nearly} and initial conditions.
This is equivalent to a non-reachability problem for the underlying
LTI system: namely the problem of whether for the LTI system~\eqref{eq:nearly-mat} the 
halfspace $H:=\{ \boldsymbol x \in \mathbb R^d : \boldsymbol e_d^\top
\boldsymbol x < 0\}$ is not reachable.

The \emph{initialised, existential} variant of the Positivity Problem
asks whether there \emph{exists} a positive  sequence satisfying a given
recurrence~\eqref{eq:nearly} with given initial conditions.  In terms of LTI systems,
this is a controlled-invariance problem: it asks whether there is a sequence of controls that
keeps the system~\eqref{eq:nearly-mat} in the halfspace $\{\boldsymbol x \in \mathbb R^d: \boldsymbol e_d^\top \boldsymbol x\geq 0\}$.  Critically for such
problems, the bang-bang principle no longer applies---one cannot
assume that the controls lie on the boundary of the control set.
It is the bang-bang principle that led us to study transcendence of the series in Theorem~\ref{thm:main} and thus 
it seems that different techniques would be 
needed for the existential variant.

The \emph{uninitialised, existential} variant of the Positivity Problem asks whether there exists a positive sequence satisfying a given nearly linear recurrence, without specifying the
initial values of the sequence.  This is a special case of the non-termination problem for linear constraint loops, which
is known to be decidable for loops in dimension two~\cite{GuilmantLO024}, but is open in dimension three.   Thus the latter variant of the Positivity Problem is decidable
for NLRS of order 2 and, we believe, open at order 3.  However, 
the techniques of~\cite{GuilmantLO024}, which are based on convex analysis, are very different from the number-theoretic approach of the present paper.

\appendix
\section{Proof of Lemma~\ref{lem:s-unit-plus}}
\label{sec:proofs}
The following is a version of the $p$-adic Subspace Theorem of
Schlickewei~\cite{Schlickewei76} and is the main tool used to prove Lemma~\ref{lem:s-unit-plus}.
We state a simplified version of the theorem in which all but one of the linear forms are 
  coordinate forms.  Throughout this section, $\mathbb K$ is a number field,
  $S \subseteq M(\mathbb K)$ is a finite set
containing all the Archimedean places of $\mathbb K$, 
  and $H$ denotes the height function on $\mathbb K$ defined in Section~\ref{sec:sums}.
  
\begin{theorem}
Given  $d\geq 2$, let $L(X_1,\ldots,X_d)$ be a linear form with
  coefficients in $\mathbb K$.  Let $i_0 \in \{1,\ldots,d\}$ be a distinguished
  index such that $X_{i_0}$ has non-zero coefficient in $L$ and let $v_0 \in S$ be a distinguished place.  
  Then for
  every $c>0$ and every $\varepsilon>0$ the set of solutions
  $(x_1,\ldots,x_d) \in (\mathcal{O}_S)^d\setminus\{0\}$ of the
  inequality
  \[ |L(x_1,\ldots,x_d)|_{v_0}\cdot  \prod_{\substack{(i,v)\in
        \{1,\ldots,d\}\times S\\(i,v) \neq (i_0,v_0)}}|x_i|_v \,
    \leq \, c H(x_1,\ldots,x_d)^{-\varepsilon} \] is contained in a finite union
  of proper linear subspaces of $\mathbb K^d$.
\label{thm:SUBSPACE}
\end{theorem}

We refer to~\cite[Section 4]{AdamczewskiBell21} for an 
account of the use of the Subspace Theorem in transcendence proofs in 
automata theory.  One of the central applications of the Subspace Theorem is in proving the following result 
about sums of $S$-units.

\begin{theorem}
 Let $v_0 \in S$.
  Given $d\geq 2$, $\alpha_1,\ldots,\alpha_d \in \mathbb K \setminus\{0\}$, $c>0$, and $\varepsilon>0$, there is a finite 
set  $U \subseteq \mathbb K$ such that for every solution of the inequality
\begin{gather*} | \alpha_1 x_1 + \cdots + \alpha_dx_d|_{v_0} < c\, H(x_1,\ldots,x_d)^{-\varepsilon} \max(|x_1|_{v_0},\ldots,|x_d|_{v_0}) 
\end{gather*}
in $S$-units $x_1,\ldots,x_d$, we have $x_i/x_j \in U$ for some $i\neq j$.
    \label{thm:s-units}
\end{theorem}

We will use Theorem~\ref{thm:s-units} as an ingredient in the proof of Lemma~\ref{lem:s-unit-plus}.
  
\begin{proof}[Proof of Lemma~\ref{lem:s-unit-plus}]
Recall that our objective is to show the existence of a finite set $U\subseteq \mathbb K$ such that 
for every solution of the inequalities
\begin{eqnarray}
\left|\alpha_0x_0+\cdots +\alpha_dx_d\right|_{v_0} &< & H(x_1,\ldots,x_d)^{-\varepsilon} \, |x_d|_{v_0}  \label{eq:s-unit-proof} \\
|x_d|_{v_0}  \prod_{v\in S\setminus \{v_0\}}|x_0|_v  & < & H(x_0,\ldots,x_{d-1})^{-\varepsilon}  \label{eq:if-equal}
\end{eqnarray}
in $S$-units $x_1,\ldots,x_d$ and $x_0 \in \mathcal O_S\setminus \{0\}$, we have $x_i/x_j\in U$ for some $i\neq j$ in $\{1,\ldots,d\}$.
To this end, consider the linear form \[ L(X_0,\ldots,X_{d-1}) := \sum_{i=0}^{d-1}\alpha_i X_i \, .\] By~\eqref{eq:s-unit-proof}, putting $c:=1+|\alpha_d|_{v_0}$, we have 
\[ |L(x_0,\ldots,x_{d-1})|_{v_0} \leq \left|\alpha_0x_0 +\cdots + \alpha_dx_d\right|_{v_0}  +\left| \alpha_d x_d \right|_{v_0} \leq c\, |x_d|_{v_0} \, .\] Since $x_1,\ldots,x_{d-1}$ are $S$-units, by the product formula and~\eqref{eq:if-equal} we have 
\begin{gather}
|L(x_0,\ldots,x_{d-1})|_{v_0} \cdot \prod_{v\in S \setminus \{v_0\}} |x_0|_v \cdot \prod_{i=1}^{d-1} \prod_{v\in S} |x_i|_v \leq c\, H(x_0,\ldots,x_{d-1})^{-\varepsilon}\, . 
\label{eq:first-sub}
\end{gather}

Applying Theorem~\ref{thm:SUBSPACE} to the form $L$, with distinguished variable $X_0$ and distinguished place $v_0$, the set of solutions $(x_0,\ldots,x_{d-1})$ of~\eqref{eq:first-sub} is contained in finitely many proper subspaces of $\mathbb K^d$.
Since every proper subspace is contained in a hyperplane,
we obtain a finite family $L_j(X_0,\ldots,X_{d-1})$, $j \in J$, of non-zero linear forms with coefficients in $\mathbb K$
such that every solution of the inequalities~\eqref{eq:s-unit-proof} and~\eqref{eq:if-equal} satisfies $L_j(x_0,\ldots,x_{d-1})=0$ for some $j \in J$.

Let $j \in J$ and consider the form $L_j$.  We consider two cases, according to whether or not $X_0$ lies in the support of $L_j$.  If not, then the equation 
$L_j(x_0,\ldots,x_{d-1})=0$ involves only $S$-units.  In this case the support of $L_j$ must include at least two variables, so
by Theorem~\ref{thm:s-units} there exists a finite set $U_j$ such that for every solution of this equation
we have $x_q/x_p \in U_j$ for some $q\neq p$ in $\{1,\ldots,d-1\}$.
Suppose, on the other hand, that $X_0$ lies in the support of $L_j$.
Note that $L_j$ is not a scalar multiple of $L'(X_0,\ldots,X_d):=\sum_{i=0}^d \alpha_i X_i$ since $X_d$ does not appear in $L_j$.
Thus, by taking a suitable linear combination of $L_j$ and $L'$, we obtain a linear form $L'_j(X_1,\ldots,X_d)$ which mentions $X_d$ but not $X_0$, and
such that for each solution of~\eqref{eq:s-unit-proof} and~\eqref{eq:if-equal} such that
$L_j(x_0,\ldots,x_{d-1})=0$ we have
\begin{gather}
|L_j'(x_1,\ldots,x_d)|_{v_0} <  H(x_1,\ldots,x_d)^{-\varepsilon} \, |x_d|_{v_0} \, .
\label{eq:above}
\end{gather}

The above is an inequality in $S$-units.  If the support of $L'_j$ includes at least two variables,  then
applying Theorem~\ref{thm:s-units} to the variables in the support of $L'_j$, there exists a finite set $U_j$ such that for every solution of~\eqref{eq:above}
we have $x_q/x_p \in U_j$ for some $q\neq p$ in $\{1,\ldots,d\}$.  Otherwise, if $L'_j$ has the form $\beta X_d$
for some $\beta\in\mathbb K \setminus\{0\}$, then~\eqref{eq:above} has the form $|\beta x_d|_{v_0} \leq H(x_1,\ldots,x_d)^{-\varepsilon}|x_d|_{v_0}$.
But this gives $H(x_1,\ldots,x_d) \leq |\beta|_{v_0}^{-1/\varepsilon}$.  By Northcott's Theorem, 
there are only finitely many projective points in $\mathbb P^{d-1}(\mathbb K)$ of bounded height.
Hence we can take $U_j$ to be the finite set of all quotients $x_p/x_q$, for $p\neq q$, over all such points.

To conclude we take the desired set $U$ to be the union of the sets $U_j$ for each $j \in J$.
\end{proof}

\section{Equality Checking}
\begin{theorem} For $j\in \{1,2\}$, define $\alpha_j:=\sum_{m=0}^\infty |a_j \lambda_j^m+\overline{a_j\lambda_j^m}|$,
where $a_j,\lambda_j \in \overline{\mathbb{Q}}$ are non-zero algebraic numbers
such that $|\lambda_j|<1$. 
Suppose that $\lambda_1,\overline{\lambda_1},\lambda_2,\overline{\lambda_2}$ are multiplicatively independent, that is,
if $\lambda_1^{k_1}\overline{\lambda_1^{k_2}}\lambda_2^{k_3}\overline{\lambda_2^{k_4}}=1$ for $k_1,\ldots,k_4\in \mathbb Z$, then
$k_1=\cdots=k_4=0$.
Then $\alpha_1-\alpha_2$ is 
transcendental. \label{thm:LI} \end{theorem}
\begin{proof}

For \(j\in\{1,2\}\) and \(m\in\mathbb N\), put $u_m^{(j)}:=a_j \lambda_j^m+\overline{a_j\lambda_j^m}$, so that
\(  \alpha_j=\sum_{m=0}^{\infty}|u_m^{(j)}| \).  
We suppose for a contradiction that \(\alpha:=\alpha_1-\alpha_2\) is algebraic.
As in Section~\ref{sec:main},
we may assume without loss of generality that \(u_m^{(j)}\neq 0\) for all $m\in\mathbb N$
and $j\in \{1,2\}$, since dropping finitely many terms from the expression for $\alpha_j$ changes the sum by an algebraic number.

The proof strategy generalises that of Theorem~\ref{thm:main}.
For $j\in\{1,2\}$ we find a family of linear recurrences, each of which is satisfied 
by the sequence of absolute values
$(|u_m^{(j)}|)_{m=0}^\infty$ outside a sparse exceptional set of indices.  We combine these two families to obtain a single family of linear recurrences,
each of which is satisfied by the sequence $(|u^{(1)}_m|-|u_m^{(2)}|)_{m=0}^\infty$ outside a sparse exceptional set.  The proof then  proceeds similarly to that of 
Theorem~\ref{thm:main}, using Lemma~\ref{lem:s-unit-plus} as the main tool.

For \(j\in\{1,2\}\), write \(  \lambda_j=|\lambda_j|\exp(2\pi i\theta_j)\)
  where $\theta_j\in(0,1)$.   
The multiplicative independence of
\(  \lambda_j,\overline{\lambda_j}\) entails that $\theta_j$ is irrational.
Let
\((q_{j,n})_{n\ge0}\)
be the sequence of convergent denominators of the continued-fraction
expansion of \(\theta_j\).  
We synchronise the two sequences of
denominators along a common scale $(Q_n)_{n=0}^\infty$ as follows. For \(n\ge0\), put $Q_n:=q_{1,n+1}$, 
\[ Q_{1,n}:=q_{1,n},\; Q_{1,n}^+:=q_{1,n+1}\quad\text{and}\quad Q_{2,n}:=q_{2,\mu_n},\; Q_{2,n}^+:= q_{2,\mu_n+1},\]
where $\mu_n:=\max\{k:q_{2,k}\leq Q_n\}$.
Then for $j\in \{1,2\}$ and $n\in \mathbb N$ we have
\(Q_{j,n}\le Q_n \le Q^+_{j,n}\).

For \(j\in\{1,2\}\) and \(n\in\mathbb N\), define
\begin{gather}
  A_{j,n}  := \lambda_j^{Q_{j,n}}  +
  \overline{\lambda_j}^{\,Q_{j,n}},
  \quad
  B_{j,n}
  :=
  (\lambda_j\overline{\lambda_j})^{Q_{j,n}}, \quad
  C_{j,n}
  :=
  (1-\lambda_j^{Q_{j,n}})
  (1-\overline{\lambda_j}^{\,Q_{j,n}}) \, .
\label{eq:abc}
\end{gather}
For \(m\ge0\), furthermore write
\begin{equation}\label{eq:two-w}
  w_{n,m}^{(j)} :=
  |u_{m+2Q_{j,n}}^{(j)}| -
  A_{j,n}|u_{m+Q_{j,n}}^{(j)}| +
  B_{j,n}|u_m^{(j)}|
\end{equation}
and let
\(
  \Delta_n^{(j)}
  :=
  \{m\ge0:w_{n,m}^{(j)}\ne0\}
  =
  \{m_{j,n,1}<m_{j,n,2}<\cdots\}
\).

Applying Propositions~\ref{prop:Delta} and~\ref{prop:gap}, for $j \in \{1,2\}$ and
all sufficiently large \(n\) we have
\begin{equation}\label{eq:two-error-expansion}
\begin{aligned}
  w_{n,m_{j,n,\ell}}^{(j)}
  &= 2a_j\xi_{j,n,\ell}
     + 2\overline{a_j}\,\overline{\xi_{j,n,\ell}},
  \\
  &\text{where }\quad
  \xi_{j,n,\ell}
  \in
  \left\{
    \pm\lambda_j^{m_{j,n,\ell}+2Q_{j,n}},
    \ 
    \pm\lambda_j^{m_{j,n,\ell}+Q_{j,n}}
    \overline{\lambda_j}^{\,Q_{j,n}}
  \right\}.
\end{aligned}
\end{equation}
In particular,
\begin{equation}\label{eq:two-error-size}
  |w_{n,m_{j,n,\ell}}^{(j)}|
  \le
  4|a_j||\lambda_j|^{m_{j,n,\ell}+2Q_{j,n}}.
\end{equation}
Moreover, for $j\in\{1,2\}$ and $\ell\geq 1$ we have
\begin{equation}\label{eq:two-four-gap}
  m_{j,n,\ell+4}-m_{j,n,\ell}
  \ge Q_{j,n}^{+}\ge Q_n
\end{equation}
and
\begin{equation}\label{eq:two-index-bounds}
  \left\lfloor\frac{\ell-1}{4}\right\rfloor
  Q_{j,n}^{+}
  \le
  m_{j,n,\ell}
  \le
  \ell Q_{j,n}^{+} \, .
\end{equation}
Finally, for each fixed \(j\) and \(\ell\ge0\), setting $ m_{j,n,0}:=0$, we have
\begin{equation}\label{eq:two-fixed-gaps}
 \lim_{n\rightarrow\infty} m_{j,n,\ell+1}-m_{j,n,\ell}
  =\infty\, .
\end{equation}

Define
\[
  \nu_{j,n}
  :=
  \sum_{m=0}^{2Q_{j,n}-1}|u_m^{(j)}|
  -
  A_{j,n}
  \sum_{m=0}^{Q_{j,n}-1}|u_m^{(j)}| \, .
\]
Since $C_{j,n}=1-A_{j,n}+B_{j,n}$,
summing \eqref{eq:two-w} over \(m\ge0\) gives
\begin{equation}\label{eq:two-single-identity}
  C_{j,n}\alpha_j
  =
  \nu_{j,n}
  +
  \sum_{m=0}^{\infty}w_{n,m}^{(j)}\, .
\end{equation}
Taking  a linear combination of the instances of~\eqref{eq:two-single-identity} for $j=1,2$ with respective 
coefficients $C_{2,n}$ and $-C_{1,n}$, and putting $\nu_n := C_{2,n}\nu_{1,n} -
  C_{1,n}\nu_{2,n}$, yields
\begin{equation}\label{eq:two-master}
  C_{1,n}C_{2,n}\alpha
 =  \nu_n + C_{2,n}  \sum_{m=0}^{\infty}w_{n,m}^{(1)} - C_{1,n}  \sum_{m=0}^{\infty}w_{n,m}^{(2)}\, .
\end{equation}
Letting $n\rightarrow\infty$ we have \(C_{j,n}\to1\).  Since $A_{j,n}\rightarrow 0$, we have also
\(\nu_{j,n}\to\alpha_j\).  It follows that
\(
  \nu_n\longrightarrow\alpha_1-\alpha_2=\alpha
\).

Recall our assumption that \(\alpha\) is algebraic and consider the number field
\[
  K
  :=
  \mathbb Q\big(
    \alpha,
    a_1,\overline{a_1},a_2,\overline{a_2},
    \lambda_1,\overline{\lambda_1},
    \lambda_2,\overline{\lambda_2}
  \big) \, .
\]
Choose a finite set of places \(S\subseteq M(K)\), containing all Archimedean
places, such that
\(
  a_1,\overline{a_1},a_2,\overline{a_2}\in\mathcal O_S
\)
and
\(
  \lambda_1,\overline{\lambda_1},
  \lambda_2,\overline{\lambda_2}
  \in\mathcal O_S^\times
\).
Let \(v_0\in S\) be the place corresponding to the usual absolute value
on \(K\subseteq\mathbb C\).
Since \(|u_m^{(j)}|\) is equal to either
\(u_m^{(j)}\) or \(-u_m^{(j)}\) we have 
 \( \nu_n\in\mathcal O_S\).
 
As in Equation~\eqref{eq:BOUND2} of Section~4.2, elementary norm estimates give a
constant \(D_1>1\), independent of \(n\), such that
\begin{equation}\label{eq:two-local-estimate}
  \prod_{v\in S\setminus\{v_0\}}
  |\nu_n|_v
  \le D_1^{Q_n}.
\end{equation}
Indeed, \(\nu_n\) is a sum of \(O(Q_n)\) terms, and all exponents
occurring in those terms are \(O(Q_n)\), since
\(Q_{j,n}\le Q_n\)\, .

Choose an integer \(h\ge1\), sufficiently large that
\begin{equation}\label{eq:two-choose-h}
 \left\lfloor \frac{h-1}{4} \right\rfloor \log(|\lambda_1|^{-1})>\log D_1\, .
\end{equation}
Next, put $\rho:=\max(|\lambda_1|,|\lambda_2|)<1$
and choose an integer \(T>h\) sufficiently large that
\begin{equation}\label{eq:two-choose-T}
  T\log(\rho^{-1})
  >
  (h+2)\log(|\lambda_1|^{-1}).
\end{equation}
Set
\(
  M_n:=TQ_n.
\)

For \(j\in\{1,2\}\), let
\(
  t_{j,n}
  :=
  \#\bigl(\Delta_n^{(j)}\cap[0,M_n)\bigr)
\).
It follows from \eqref{eq:two-four-gap} that
\(
  t_{j,n}\le 4T\).
After passing to an infinite subsequence of \(n\), we may therefore
assume that \(t_j:=t_{j,n}\) is independent of \(n\), for each
\(j\in\{1,2\}\). Passing to a further subsequence, we may also assume
that, for each \(j\) and each \(1\le \ell\le t_j\), the sign and the
choice of monomial in \eqref{eq:two-error-expansion} are independent of \(n\).
Notice that
\[
  m_{1,n,h}\le hQ_{1,n}^+=hQ_n<M_n,
\]
so \(t_1\ge h\).

Truncating the infinite sums on the right-hand side of~\eqref{eq:two-master} at \(M_n\), define
\begin{equation}\label{eq:two-truncated-form}
\begin{split}
  \mathcal L_n
  :={}&
  \nu_n-C_{1,n}C_{2,n}\alpha
  +C_{2,n}
  \sum_{\ell=1}^{t_1}
  w_{n,m_{1,n,\ell}}^{(1)} - C_{1,n} \sum_{\ell=1}^{t_2}
  w_{n,m_{2,n,\ell}}^{(2)}
\end{split}
\end{equation}
The above definition of $\mathcal L_n$ and Equation~\eqref{eq:two-master} give
\begin{equation}\label{eq:two-truncated-tail}
  \mathcal L_n
  = - C_{2,n}
  \sum_{\substack{m\in\Delta_n^{(1)}\\m\ge M_n}}
  w_{n,m}^{(1)} + C_{1,n}
  \sum_{\substack{m\in\Delta_n^{(2)}\\m\ge M_n}}
  w_{n,m}^{(2)} \, .
\end{equation}
Since \(C_{j,n}\to1\), the numbers \(C_{j,n}\) are bounded
with respect to \(|\cdot|_{v_0}\). Hence, using the bound for $|w_{n,m}^{(j)}|$ in \eqref{eq:two-error-size} and bounding each
sum in~\eqref{eq:two-truncated-tail} by the corresponding full geometric sum, there is a
constant \(D_0>0\) such that
\begin{equation}\label{eq:two-tail-estimate}
  |\mathcal L_n|
  \le
  D_0(
  |\lambda_1|^{M_n+2Q_{1,n}}+|\lambda_2|^{M_n+2Q_{2,n}})
  \le
  D_0\rho^{TQ_n}.
\end{equation}

As a next step, we rewrite the product terms 
\begin{gather} 
\label{eq:exponents}
C_{1,n}C_{2,n},\quad C_{1,n}w_{n,m_{2,n,\ell}}^{(2)}, \quad w^{(1)}_{n,m_{1,n,\ell}}C_{2,n} 
\end{gather}
appearing in the definition of $\mathcal L_n$ in~\eqref{eq:two-truncated-form} as polynomials in 
$\lambda_1,\overline{\lambda_1},\lambda_2,\overline{\lambda_2}$.   
To this end, define the exponent of a monomial
\[X=\lambda_1^{k_1}\overline{\lambda_1^{k_2}}\lambda_2^{k_3}\overline{\lambda_2^{k_4}},\quad (k_1,\ldots,k_4) \in \mathbb N^4\, , \] to be
$\mathbf e(X):=(\mathbf e_1(X),\mathbf e_2(X))$, where $\mathbf e_1(X):=(k_1,k_2)$ and $\mathbf e_2(X):=(k_3,k_4)$.

For \(j\in\{1,2\}\) and \(n\in \mathbb N\), if $X$ is a monomial of $C_{j,n}$ then by~\eqref{eq:abc} we have
$\mathbf e_j(X) \in \mathcal C_{j,n}$, where 
\[
 \mathcal C_{j,n}
  :=
  \left\{
    (0,0),\,
    (Q_{j,n},0),\,
    (0,Q_{j,n}),\,
    (Q_{j,n},Q_{j,n})
  \right\} \, .
\]
Similarly, by Equation~\eqref{eq:two-error-expansion}, if $X$ is a monomial of $w_{n,m_{j,n,\ell}}^{(j)}$, then $\mathbf e_j(X)\in \mathcal W_{j,n,\ell}$, where
\[
  \mathcal W_{j,n,\ell}
  :=
  \left\{
    (m_{j,n,\ell}+2Q_{j,n},0),\,
    (m_{j,n,\ell}+Q_{j,n},Q_{j,n}),\,
    (0,m_{j,n,\ell}+2Q_{j,n}),\,
    (Q_{j,n},m_{j,n,\ell}+Q_{j,n})
  \right\}.
\]
Thus if $X$ is a  monomial of one of the product terms~\eqref{eq:exponents}, we have $\mathbf e(X)\in\mathcal E_n$, where
\[ \mathcal E_n:=
  (\mathcal C_{1,n}\times\mathcal C_{2,n})\cup 
  \bigcup_{1\le\ell\le t_1} (\mathcal W_{1,n,\ell}\times\mathcal C_{2,n})
  \cup \bigcup_{1\le\ell\le t_2}
  (\mathcal C_{1,n}\times\mathcal W_{2,n,\ell})\, .
\]

\begin{claim}
For $n$ sufficiently large, the union defining $\mathcal E_n$ is disjoint and 
\[ \displaystyle \lim_{n\rightarrow \infty} \, \displaystyle \min_{\substack{\boldsymbol e,\boldsymbol e' \in \mathcal E_n\\
\boldsymbol e\neq \boldsymbol e'}} 
\|\boldsymbol e-\boldsymbol e'\|_\infty =\infty \, .\]
\label{claim:two-exponent-separation}
\end{claim}
\begin{proof}
We have the following observations about $\boldsymbol e,\boldsymbol e'\in \mathbb N^2$, $\boldsymbol e\neq \boldsymbol e'$,
and $j\in\{1,2\}$.  
\begin{enumerate}
    \item if $\boldsymbol e,\boldsymbol e' \in \mathcal C_{j,n}$ then $\|\boldsymbol e-\boldsymbol e' \|_\infty = Q_{j,n}$.
    \item if $\boldsymbol e \in \mathcal C_{j,n}$ and $\boldsymbol e' \in\mathcal W_{j,n,\ell}$
for some $\ell\in\{1,\ldots,t_j\}$, then 
\( \|\boldsymbol e-\boldsymbol e' \|_\infty \geq \min(Q_{j,n},m_{j,n,\ell})\).
\item if $\boldsymbol e \in \mathcal W_{j,n,\ell}$ and  
$\boldsymbol e' \in \mathcal W_{j,n,\ell'}$ for some $\ell\neq \ell'$ lying in $\{1,\ldots,t_j\}$, then 
\[ \|\boldsymbol e-\boldsymbol e' \|_\infty \geq \min(Q_{j,n},m_{j,n,\ell},|m_{j,n,\ell} - m_{j,n,\ell'}|) \, .
\]
\item if  $\boldsymbol e, \boldsymbol e' \in \mathcal W_{j,n,\ell}$ for some $\ell \in \{1,\ldots,t_j\}$,
then 
\(\|\boldsymbol e-\boldsymbol e' \|_\infty \geq \min(Q_{j,n},m_{j,n,\ell})\).
\end{enumerate}
It remains to note that by~\eqref{eq:two-fixed-gaps} and the fact that $\lim_{n\rightarrow\infty} Q_{j,n}=\infty$, the lower bounds 
in the previous four items diverge to $\infty$.
\end{proof}

It follows from Claim~\ref{claim:two-exponent-separation} that, after dropping finitely many further
values of \(n\), the monomials occurring in different product terms in~\eqref{eq:exponents} are distinct. We may therefore
write
\begin{equation}\label{eq:two-linear-form}
  \mathcal L_n=L(\Lambda_n),
  \qquad
  \Lambda_n=(\Lambda_{n,0},\ldots,\Lambda_{n,D}),
\end{equation}
where
\[
  L(X_0,\ldots,X_D)
  =
  X_0+\beta_1X_1+\cdots+\beta_DX_D
\]
is a fixed linear form over \(K\), every \(\beta_i\) is non-zero, \(\Lambda_{n,0}:=\nu_n\), and
for $i\in\{1,\ldots,D\}$, \(\Lambda_{n,i}\) is a monomial in
\(
  \lambda_1,\overline{\lambda_1},
  \lambda_2,\overline{\lambda_2}
\)
with exponent in $\mathcal E_n$.
The contribution of the $h$-th exceptional term contains the two
distinct monomials $\xi_{1,n,h}$ and
$\overline{\xi_{1,n,h}}$, with non-zero coefficients, so $D\geq 2$.

As in Equation~\eqref{eq:HEIGHT} of Section~4.2, elementary height estimates yield a
constant \(D_2>1\), independent of \(n\), such that
\begin{equation}\label{eq:two-height-estimate}
  H(\Lambda_n)
  \le
  D_2^{(T+4)Q_n}.
\end{equation}
Indeed, by construction, the exponent vector of every $S$-unit
coordinate of $\Lambda_n$ belongs to $\mathcal E_n$.  Such a monomial has total exponent at most 
\[
  M_n+2Q_{j,n}+2Q_{3-j,n}
  \le
  (T+4)Q_n.
\]
After increasing \(D_2\), if necessary, the same bound accounts for
the height of \(\nu_n\) and for the fixed number of coordinates.

We now choose the distinguished \(S\)-unit. From the contribution of
the \(h\)-th exceptional term of the first sequence and the constant
term of
\(C_{2,n},
\)
let \(\Theta_n\) denote the unsigned monomial obtained by absorbing the
sign in \eqref{eq:two-error-expansion} into its coefficient. Thus
\begin{equation}\label{eq:two-theta-type}
  \Theta_n
  \in
  \left\{
    \lambda_1^{m_{1,n,h}+2Q_{1,n}},
    \lambda_1^{m_{1,n,h}+Q_{1,n}}
    \overline{\lambda_1}^{\,Q_{1,n}}
  \right\}.
\end{equation}
Its coefficient in \(L\) is \(\pm2a_1\), and Claim~\ref{claim:two-exponent-separation}
shows that this coordinate is not cancelled by any other term.
Moreover,
\[
  |\Theta_n|
  =
  |\lambda_1|^{m_{1,n,h}+2Q_{1,n}}.
\]
Since \(Q_{1,n}^+=Q_n\), Equation
\eqref{eq:two-index-bounds}  gives
\begin{equation}\label{eq:two-theta-bounds}
  |\lambda_1|^{(h+2)Q_n}
  \le
  |\Theta_n|
  \le
  |\lambda_1|^{\lfloor \frac{h-1}{4} \rfloor Q_n}.
\end{equation}
Reorder the \(S\)-unit coordinates so that
\(
  \Lambda_{n,D}=\Theta_n
\).

Put
\[
  \delta_1
  :=
  \left\lfloor \frac{h-1}{4}\right\rfloor \log(|\lambda_1|^{-1})-\log D_1>0
\quad
\text{and}
\quad
  \delta_2
  :=
  T\log(\rho^{-1})
  -(h+2)\log(|\lambda_1|^{-1})>0.
\]
By \eqref{eq:two-local-estimate} and the upper bound in
\eqref{eq:two-theta-bounds},
\begin{equation}\label{eq:two-second-smallness}
  |\Theta_n|
  \prod_{v\in S\setminus\{v_0\}}
  |\nu_n|_v
  \le
  e^{-\delta_1Q_n}.
\end{equation}
Likewise, by \eqref{eq:two-tail-estimate} and the lower bound in
\eqref{eq:two-theta-bounds},
\begin{equation}\label{eq:two-first-smallness-pre}
  \frac{|L(\Lambda_n)|}{|\Theta_n|}
  \le
  D_0e^{-\delta_2Q_n}.
\end{equation}
Thus, for all sufficiently large \(n\), the right-hand side of
\eqref{eq:two-first-smallness-pre} is at most
\(e^{-\delta_2Q_n/2}\).

Let
\(
  B_0:=(T+4)\log D_2
\)
and choose
\[
  0<\varepsilon
  <
  \frac{1}{2B_0}
  \min\left\{\delta_1,\frac{\delta_2}{2}\right\}.
\]
Since deleting coordinates from a tuple can only decrease its projective height,
\eqref{eq:two-height-estimate}--\eqref{eq:two-first-smallness-pre}
imply, for all sufficiently large \(n\),
\begin{equation}\label{eq:two-lemma-first}
  |L(\Lambda_n)|
  <
  H(\Lambda_{n,1},\ldots,\Lambda_{n,D})^{-\varepsilon}
  |\Lambda_{n,D}|
\end{equation}
and
\begin{equation}\label{eq:two-lemma-second}
  |\Lambda_{n,D}|
  \prod_{v\in S\setminus\{v_0\}}
  |\nu_n|_v
  <
  H(\nu_n,\Lambda_{n,1},\ldots,\Lambda_{n,D-1})^{-\varepsilon}.
\end{equation}

Suppose that $\nu_n\neq 0$ for infinitely many $n$.  Then, after restricting to
this infinite subsequence, we apply Lemma~\ref{lem:s-unit-plus} 
with \(x_0=\nu_n\) and distinguished
\(S\)-unit
\(
  x_D=\Lambda_{n,D}=\Theta_n
\).
We conclude that there is a finite set \(U\subseteq K\) such that, for
every sufficiently large \(n\), there are distinct indices
\(r,s\in\{1,\ldots,D\}\) for which
\(
  \frac{\Lambda_{n,r}}{\Lambda_{n,s}}\in U \, .
\)
Suppose, on the other hand, that  $\nu_n=0$ for infinitely many $n$.
On this subsequence,~\eqref{eq:two-lemma-first} is an inequality
in $S$-units alone, and Theorem~\ref{thm:s-units} yields the same
conclusion: two distinct $S$-unit coordinates have quotient in a fixed
finite set.  In either case,
since there are only finitely many pairs \((r,s)\) and finitely many
elements of \(U\), there exist fixed distinct indices \(r,s\) and a
fixed \(u\in U\) such that
\begin{equation}\label{eq:two-fixed-ratio}
  {\Lambda_{n,r}}/{\Lambda_{n,s}}=u
\end{equation}
for infinitely many \(n\).   

Choose any two indices \(n,n'\) satisfying
\eqref{eq:two-fixed-ratio}. Dividing the two resulting equalities gives
\[
  \lambda_1^{k_1}
  \overline{\lambda_1^{k_2}}
  \lambda_2^{k_3}
  \overline{\lambda_2^{k_4}}
  =1,
\]
where
\[
  (k_1,k_2,k_3,k_4)
  =
  \bigl(
    \mathbf e(\Lambda_{n,r})-\mathbf e(\Lambda_{n,s})
  \bigr)
  -
  \bigl(
    \mathbf e(\Lambda_{n',r})-\mathbf e(\Lambda_{n',s})
  \bigr).
\]
By the assumed multiplicative independence of
\(
  \lambda_1,\overline{\lambda_1},
  \lambda_2,\overline{\lambda_2},
\)
all four exponents are zero. Hence
\(
  \mathbf e(\Lambda_{n,r})-\mathbf e(\Lambda_{n,s})
\)
is constant on the infinite set of indices satisfying
\eqref{eq:two-fixed-ratio}. This contradicts
Claim~\ref{claim:two-exponent-separation}, because \(r\ne s\).
The contradiction shows that \(\alpha\) is not algebraic.
\end{proof}

\begin{corollary}
Under the hypotheses of Theorem~\ref{thm:LI} one can decide whether
$\alpha_1<\alpha_2$ or $\alpha_2<\alpha_1$
by computing certified
approximations of the two exponentially convergent series. Theorem~\ref{thm:LI} 
guarantees that the approximation intervals eventually become
disjoint.
\end{corollary}

\bibliography{bibtex}
\end{document}